\documentclass[11pt]{article}
\usepackage{latexsym} 
\usepackage{amssymb}
\usepackage{amsmath}
\usepackage{amsthm}
\newtheorem{thm}{Theorem}[section]
\newtheorem{lem}[thm]{Lemma}

\newtheorem{rem}[thm]{Remark}
\newtheorem{dfi}[thm]{Definition}
\newtheorem{exm}[thm]{Example}
\newtheorem{pro}[thm]{Proposition}
\newtheorem{assumption}{Assumption}

\newcommand{\E}{{\mathcal E}}
\newcommand{\RR}{{\mathbb R}}

\newcommand{\FC}{{\cal FC}^{\infty}_b}
\newcommand{\ep}{\varepsilon}
\newcommand{\la}{\lambda}

\newcommand{\Ricbar}{\overline{\text{\rm Ric}(\gamma)}_t}

\newcommand{\DVbar}{\overline{\nabla V_{y_0}^{\la}(t,\gamma)}_t}
\newcommand{\grad}{\text{\rm grad}}
\newcommand{\cxy}{c_{x_0,y_0}}
\newcommand{\cyx}{c_{y_0,x_0}}
\renewcommand{\H}{{\rm H}}
\newcommand{\Ho}{{\rm H}_0}
\newcommand{\Rreverse}{R^{\leftarrow}}
\renewcommand{\Im}{{\rm Image}}
\newcommand{\esssup}{{\rm esssup}}
\newcommand{\ldirichlet}{e^{\lambda}_{Dir,1,{\mathcal D}}}
\newcommand{\dirichlet}{e^{\lambda}_{Dir,2,{\mathcal D}}}
\newcommand{\Ltwo}{{\rm L^2}}
\newcommand{\Ltwoo}{{\rm L^2_0}}
\newcommand{\chio}{\chi_{0,\kappa}}
\numberwithin{equation}{section}
\begin{document}
\newcounter{aaa}
\newcounter{bbb}
\newcounter{ccc}
\newcounter{ddd}
\newcounter{eee}
\newcounter{fff}
\newcounter{xxx}
\newcounter{xxxiv}
\newcounter{xvi}
\newcounter{x}
\setcounter{aaa}{1}
\setcounter{bbb}{2}
\setcounter{ccc}{3}
\setcounter{ddd}{4}
\setcounter{fff}{6}
\setcounter{eee}{32}
\setcounter{xxx}{10}
\setcounter{xxxiv}{34}
\setcounter{xvi}{16}
\setcounter{x}{38}
\title
{Asymptotics of spectral gaps
on loop spaces\\
over a class of Riemannian manifolds}
\author{Shigeki Aida\\
Mathematical Institute\\
Tohoku University,
Sendai, 980-8578, JAPAN\\
e-mail: aida@math.tohoku.ac.jp}
\date{}
\maketitle
\begin{abstract}
We prove the existence of spectral gaps of
Ornstein-Uhlenbeck operators on loop spaces over a
class of Riemannian manifolds which
include hyperbolic spaces.
This is an alternative proof and an extension of a result in
Chen-Li-Wu in J. Funct. Anal. 259 (2010), 1421-1442.
Further, we study the asymptotic behavior of the
spectral gap.

\end{abstract}

\section{Introduction}
Let $E$ be a smooth non-negative function on a Riemannian manifold $X$.
Let $\la$ be a positive number and
consider a weighted probability measure 
$d\nu^{\la}(x)=Z_{\la}^{-1}e^{-\la E(x)}dx$ on $X$,
where $Z_{\la}$ denotes the normalized constant 
and $dx$ denotes the Riemannian volume.
We consider a Dirichlet form on $L^2(X,d\nu^{\la})$
such that
$$
\E^{\la}(F,F)=\int_{X}|\nabla F(x)|^2d\nu^{\la}(x),
$$
where $\nabla$ denotes the Levi-Civita covariant derivative.
Under mild assumptions on $E$ and the Riemannian metric,
$1\in {\rm D}(\E^{\la})$ and
the corresponding lowest eigenvalue $e^{\la}_1$
of the generator of the Dirichlet form
is $0$.
The spectral gap $e^{\la}_2$ of $\E^{\la}$ is defined by
\begin{align}
 e_2^{\la}&=\inf\Bigl\{\E^{\la}(F,F)~\Big |~\|F\|_{L^2(\nu^{\la})}=1,
\int_XF(x)d\nu^{\la}(x)=0\Bigr\}.
\end{align}
The study on the estimate and the asymptotic behavior of
$e^{\la}_2$
as $\la\to\infty$ is an interesting and important subject.
In this problem,
one of the simplest cases is the following:
\begin{enumerate}
\item[(i)] $E$ has a unique minimum point $c_0$ and 
there are no critical points other than $c_0$,
\item[(ii)] the Hessian of $E$ at $c_0$ is non-degenerate.
\end{enumerate}
In this case, under some additional technical assumptions,
it holds that 
$\lim_{\la\to\infty}\frac{e_2^{\la}}{\la}=\sigma_1$,
where $\sigma_1$ is the lowest eigenvalue of the 
Hessian of $E$ at $c_0$.
When $X=\RR^N$ and $E(x)=\frac{|x|^2}{2}$,
the generator of the Dirichlet form is called 
the Ornstein-Uhlenbeck operator(=OU operator) and the spectral set is 
completely known.

We are interested in the case where $X$ is 
an ``infinite dimensional Riemannian manifold''
and $\nu^{\la}$ is a probability measure on it.
Let us explain our model.
Let $(M,g)$ be an $n$-dimensional complete Riemannian manifold.
Let $x_0, y_0\in M$ and consider 
a space of continuous paths
$P_{x_0}(M)=C([0,1]\to M~|~\gamma(0)=x_0)$
and its subset
$P_{x_0,y_0}(M)=\{\gamma\in P_{x_0}(M)~|~\gamma(1)=y_0\}$.
Our $X$ is $P_{x_0}(M)$ or $P_{x_0,y_0}(M)$ and
$\nu^{\la}$ is the (pinned) Brownian motion
measure.
The transition probability of the Brownian motion is given by
$p(t/\la,x,y)$, where $p(t,x,y)$
denotes the heat kernel of the diffusion semigroup
$e^{t\Delta/2}$ and
$\Delta$ is the Laplace-Bertlami operator.
In many problems, we use the following 
heuristically
appealing path integral expression,
$$
d\nu^{\la}(\gamma)=\frac{1}{Z_{\la}}
\exp\left(-\la E(\gamma)\right)
d\gamma,
$$
where $E(\gamma)$ is the energy of
path $\gamma$
and $d\gamma$ is the ``infinite dimensional Riemannian measure''.
Of course, the energy function cannot be defined on the continuous 
path spaces on which the (pinned) Brownian motion measures exist
and there do not exist the ``Riemannian measures'' on 
the infinite dimensional spaces.
We refer the reader to \cite{andersson-driver, lim, laetsch} for some
rigorous study of the path integral.
On the other hand, by using an $H$-derivative $D$ on $X$
(see the definition in Section~3),
we can define a Dirichlet form $\E^{\la}$ on $L^2(X,d\nu^{\la})$.
Our interest is in the study of the spectral gap of
$\E^{\la}$.
Since
the triple $(X,\nu^{\la}, \E^{\la})$ is formally 
an infinite dimensional analogue of the finite dimensional one,
we may conjecture some results on the asymptotics of the spectral gap.

In the case where $X=P_{x_0}(M)$, the critical point of
$E$ on the subset of $H^1$ paths is just a
constant path and this problem corresponds to
the simplest case which we explained.
Fang~\cite{fang} proved the existence of the
spectral gap by establishing the COH(=Clark-Ocone-Haussmann) formula
for functions on $X=P_{x_0}(M)$.
Also it is not difficult to prove that
$\lim_{\la\to\infty}\frac{e_2^{\la}}{\la}=1$
by using the COH formula.
We prove this in Section 3.
Here note that the Hessian of $E$
at the constant path is identity.
On the other hand, if $X$ is the pinned space $P_{x_0,y_0}(M)$,
the set of critical points of the functional 
$E$
on the set of $H^1$ paths of $P_{x_0,y_0}$ is 
the set of geodesics.
Therefore,
by an analogy of finite dimensional cases,
one may expect that the asymptotic behavior of the low-lying spectrum
of the generator of $\E_{\la}$ is related to the set of
the geodesics in this case.
However, it is not even easy to find examples of Riemannian manifolds
on which loop spaces the spectral gaps exist.
In fact, Eberle~\cite{eberle1} gave an example of a Riemannian manifold
which is diffeomorphic to
a sphere over which there is no spectral gap on the loop space.
At the moment, there are no examples of loop spaces over 
simply connected compact Riemannian
manifold for which the spectral gap exists.

If $M$ is a Riemannian manifold with a pole $y_0$,
the situation is simpler.
In this case, the function $E$ defined on the 
$H^1$ subset of $P_{x_0,y_0}(M)$
satisfies the above mentioned 
assumptions (i) and (ii).
The author proved the existence of spectral
gap in that case under additional strong assumptions on
the Riemannian metric in \cite{aida-precise}.
Unfortunately, the assumption is not valid for
hyperbolic spaces.
The existence of the spectral
gap on loop spaces over hyperbolic spaces
was proved by Chen-Li-Wu~\cite{clw1} for the first time
(see \cite{clw11} also).
They used results in \cite{aida-coh, cgg}.
We give an alternative proof of their result
and prove that $\lim_{\la\to\infty}\frac{e_2^{\la}}{\la}=\sigma_1$,
where $\sigma_1$ is the spectrum bottom of the Hessian of $E$ at the
unique geodesic for a certain class of Riemannian manifolds.

Now let us recall a rough idea how to prove the asymptotic behavior of
$e_2^{\la}$ under the assumptions (i), (ii) when 
$X$ is a finite dimensional space.
By the unitary transformation $M_{\la} : F(\in L^2(d\nu^{\la}))\mapsto
F\left(Z_{\la}^{-1}e^{-\la E}\right)^{1/2}(\in L^2(dx))$, the problem 
is changed to determine the limit of the gap of spectrum
of a Schr\"odinger operator.
In this context, $\la\to\infty$ corresponds to
the semi-classical limit of a physical system.
In a small neighborhood of $c_0$, the Schr\"odinger operator
can be approximated by a harmonic oscillator and we obtain 
the main term of the divergence of $e_2^{\la}$.
As for outside the neighborhood,
the potential function is very large and it has nothing to do with
low energy
part of the operator.
In the present infinite dimensional problems,
we cannot use the unitary transformation
since there does not exist Riemannian volume measure and 
the function $E$ cannot be defined on the whole space $X$.
Moreover, there are difficulties in
the proof of each parts, (a) Local estimate in a neighborhood
$U(c_0)$ of the minimizer,
(b) Estimate outside $U(c_0)$.
In the problem (a), one may think that 
the problem can be reduced to
a Gaussian measure case by a certain 
``local diffeomorphism''.
A natural candidate of the local diffeomorphism is 
an It\^o map.
Certainly, the mapping is measure preserving
but the derivative of the mapping does not behave well
because of the irregularity of the Brownian paths
~\cite{driver1, cruzeiro-malliavin, elworthy-li1}.
In problem (b), it is not clear how to use ``the potential function is
big''
outside $U(c_0)$.
To solve these problems, we use COH formula and
a logarithmic Sobolev inequality on $X$.
Clearly, it is more interesting to
consider the cases where there are two or more local minimum points of
$E$.
We refer the reader to \cite{hks, hf} and references therein
for finite dimensional cases.
Also we note that Eberle~\cite{eberle3} 
studied such a problem on certain approximate spaces 
of loop spaces.

The paper is organized as follows.
We already explained a rough idea of a proof of the
asymptotic behavior 
of $e^{\la}_2$.
In Section 2, we give a different proof
based on a log-Sobolev inequality.
Our proof for loop spaces is a modification of the proof.
Also we explain the difficulty of the proof
in the case of loop spaces.

In Section 3,
we prepare necessary definitions and
lemmas and explain our main theorems for
$P_{x_0,y_0}(M)$.
In this case, the minima $c_0$ is the minimal geodesic 
$c_{x_0,y_0}$ between $x_0$ and $y_0$.
As we explained, we need local analysis in a neighborhood of
$c_{x_0,y_0}$ of
the generators of Dirichlet forms.
Thus we consider an OU operator
with Dirichlet boundary condition on
a small neighborhood ${\cal D}$ of the minimal geodesic in a loop space
over a Riemannian manifold.
We define the 
generalized second lowest eigenvalue 
$\dirichlet$ of the Dirichlet
Laplacian and determine the
asymptotic behavior of
$\dirichlet$
in our first main theorem (Theorem~\ref{main theorem 1}).

In the second main theorem (Theorem~\ref{main theorem 2}), we consider 
a rotationally symmetric 
Riemannian manifold $M$ with a pole $y_0$
and a loop space $P_{x_0,y_0}(M)$, where
$x_0$ is an arbitrary point of $M$.
Under certain assumptions on the Riemannian metric,
we prove the existence of the spectral gap and 
determine the asymptotic behavior of
$e^{\la}_2$.
The class of Riemannian manifolds includes the hyperbolic spaces.
Actually, the same result as in the second main theorem
holds true under the validity of
a certain log-Sobolev inequality and a tail estimate of
a certain random variable describing the size of
$\gamma$.
The log-Sobolev inequality can be proved by a
COH-formula on $P_{x_0,y_0}(M)$.
The diffusion coefficient of the Dirichlet form in the
log-Sobolev inequality is unbounded and it is still an
open problem whether a log-Sobolev inequality with 
a bounded coefficient
holds on a loop space over a hyperbolic space .

In this paper, the COH formula plays a crucial role.
Let us recall what COH formula is.
Let $F$ be an $L^2$ random variable on $P_{x_0}(M)$.
By the It\^o theorem,
$F-E^{\nu^{\la}}[F]$ can be represented as a
stochastic integral with respect to 
the Brownian motion $b$ which is obtained as an
anti-stochastic development of $\gamma$ to
$\RR^n$ (\cite{hsu}).
The COH formula gives an explicit form of the integrand
as a conditional expectation 
of the $H$-derivative $DF$.
As we noted, Fang proved the COH formula on
$P_{x_0}(M)$ when $M$ is a compact Riemannian manifold.
But it is not difficult to prove the same formula
for more general Riemannian manifold
(see Lemma~\ref{fang COH formula}).
In the case of $P_{x_0,y_0}(M)$, it is necessary
to consider a Brownian motion $w$ under the
pinned measure which is obtained by adding a singular drift 
to $b$.
The singular drift 
is defined by a logarithmic derivative
of $p(t,y_0,z)$.
For this, see Lemma~\ref{coh formula} and \cite{aida-coh, aida-coh2}.
In both cases of $P_{x_0}(M)$ and $P_{x_0,y_0}(M)$,
the integrand in the COH formula is the conditional expectation of
the quantity $A(\gamma)_{\la}(DF')$, where $A(\gamma)_{\la}$ is a
certain bounded linear operator depending on
the path $\gamma$ and $\la$.
$A(\gamma)_{\la}$ for $P_{x_0}(M)$
is defined by the Ricci curvature and
the operator norm is uniformly bounded for large
$\la$.
On the other hand,
in the case of $P_{x_0,y_0}(M)$,
the definition of 
$A(\gamma)_{\la}$ contains the Hessian of the heat kernel,
$\nabla_z^2\log p(t/\la,y_0,z)$ $(0<t\le 1)$
because the stochastic differential equation of $\gamma$ 
contains the singular drift term of the logarithmic derivative of
the heat kernel.
To control
this term,
we need results for a short time behavior of
$\lim_{t\to 0}\nabla_z^2\log p(t,x,z)$ which were studied for the first
time by Malliavin and Stroock~\cite{ms}
(see (\ref{loghessian 0}) and Lemma~\ref{gong-ma}).
In view of this, it is easier to study
the spectral gap for $P_{x_0}(M)$ than that for $P_{x_0,y_0}(M)$.
In the final part of this section, we prove 
$\lim_{\la\to\infty}\frac{e_2^{\la}}{\la}=1$ for $P_{x_0}(M)$.


In order to show the precise asymptotics of
$\dirichlet$ and $e^{\la}_2$, we need to
identify $A(\cxy)_{\infty}=\lim_{\la\to\infty}A(\cxy)_{\la}$.
This is necessary for local analysis near $\cxy$.
In Section 4, first we formally show that
$A(\cxy)_{\infty}$
is an operator which is defined by
the Hessian of the square of the distance function
$k(z)=\frac{d(z,y_0)^2}{2}$.
After that we prove a key relation between
the Hessian of the energy function $E$ at $\cxy$ and
$A(\cxy)_{\infty}$.
In that proof, Jacobi fields along the geodesic
play an important role.

In Section 5, we prove Theorem~\ref{main theorem 1}.
The proof ${\rm LHS}\le {\rm RHS}$ in $(\ref{main theorem 1 identity})$
relies on an explicit representation $(\ref{representation of e2})$ of 
$\dirichlet$ by the unique eigenfunction (ground state)
$\Psi_{\la}$
associated with the first eigenvalue of the Dirichlet Laplacian.
By using this representation and a trial function,
we prove the upper bound.
The trial function is closely related with
``eigenfunctions'' associated with the bottom of the spectrum
of the Hessian of the energy function $E$ at $\cxy$.

As already mentioned, we need to study 
$A(\cxy)_{\infty}$.
In addition,
we need to show that $A(\gamma)_{\la}$ can be approximated by
$A(\cxy)_{\infty}$ when $\gamma$ is close to $\cxy$ and
$\la$ is large.
This is correct but not trivial because
$A(\gamma)_{\la}$ is defined by solutions of It\^o's
stochastic differential equations driven by $b$
and the solution mappings are not continuous
in usual topology such as the uniform convergence topology.
Actually the solution mappings are continuous
in the topology of rough paths.
Thus, we need to
apply rough path analysis to our problem.
Note that the law of $b$ under the pinned measure 
is singular with respect to the Brownian motion measure.
However, the probability distribution of $b$ does not charge
the slim sets in the sense of Malliavin.
Hence, we need to consider Brownian rough paths 
for all Brownian paths except a slim set (\cite{aida-loop group}).
After preparation of necessary estimates from rough paths
(see Lemma~\ref{lemma from rough path}), we prove Theorem~\ref{main
theorem 1}.

In Section 6, we prove 
the existence of the spectral gap in
a certain general setting as in \cite{clw1}.
This third main theorem (Theorem~\ref{main theorem 3}) 
implies the first half of the statement in
Theorem~\ref{main theorem 2}.
In Section 7, we complete the proof of
Theorem~\ref{main theorem 2}.

\section{A proof in $\RR^N$ and some remarks}

In this section, we show a proof of the asymptotics
$\lim_{\la\to\infty}\frac{e^{\la}_2}{\la}=\sigma_1$
on $\RR^N$
under the validity of a log-Sobolev inequality.
Our proof for $P_{x_0,y_0}(M)$ is
a suitable modification of this proof.
In this section, $D$ stands for the usual Fr\'echet derivative on $\RR^N$.

Let $E$ be a non-negative $C^{\infty}$ function on $\RR^N$
and suppose the following (1), (2), (3), (4).

\begin{enumerate}
 \item[(1)] $E(0)=0$ and $0$ is the unique minimum point and
$D^2E(0)>0$.
Further $\liminf_{|x|\to\infty}E(x)>0$.
\item[(2)] Let $\la>0$.
Suppose that $e^{-\la E(x)}$ is an integrable function and
define a probability measure,
\begin{align}
 \nu^{\la}(dx)=Z_{\la}^{-1}e^{-\la E(x)}dx,
\end{align}
where $Z_{\la}=\int_{\RR^N}e^{-\la E(x)}dx$.
\item[(3)] 
Let $\E^{\la}(F,F)=\int_{\RR^N}|DF(x)|^2d\nu^{\la}(x)$,
where $F\in C^{\infty}_0(\RR^N)$.
Also let $\E^{\la}$ denote the Dirichlet form which
is the closure of the closable form.
It holds that $|x|, 1\in {\rm D}(\E^{\la})$
and $\E^{\la}(1,1)=0$ for all $\la>0$.
The notation $|\cdot|$ denotes the usual Euclidean norm.
\item[(4)]
There exists a constant $C>0$ such that the following
log-Sobolev inequality holds:
\begin{align}
\int_{\RR^N} 
F(x)^2\log \left(F(x)^2/\|F\|_{L^2(\nu_{\la})}^2\right)
d\nu^{\la}(x)
&\le
\frac{C}{\la}\E^{\la}(F,F),\quad
F\in {\rm D}(\E^{\la}).
\label{LSI RN}
\end{align}
\end{enumerate}

Clearly the spectral bottom $e_1^{\la}$ of the Dirichlet form
$\E^{\la}$ is $0$.
Under the above assumptions, we prove that

\begin{thm}\label{RN}
Let $e_2^{\la}$ be the spectral gap of $\E^{\la}$.
Then
 \begin{align}
  \lim_{\la\to\infty}\frac{e_2^{\la}}{\la}=\sigma_1,
 \end{align}
where $\sigma_1$ denotes the smallest eigenvalue of
the matrix $D^2E(0)$.
\end{thm}

The log-Sobolev inequality (\ref{LSI RN})
implies the bound
$e_2^{\la}\ge 2\la/C$ for all $\la$.
So it holds that $C\sigma_1\ge 2$.
Note that the assumption in the above is very strong and
we cannot say the result is ``nice''.

\begin{proof}
We prove the lower bound estimate
$\liminf_{\la\to\infty}\frac{e^{\la}_2}{\la}\ge \sigma_1$.
By the assumptions (1) and (2), we have for any $r>0$ there exists $K_{r}$ and
$M_{r}$ such that 
\begin{align}
\nu^{\la}\left(|x|\ge r\right)\le 
K_re^{-\la M_r}\qquad \mbox{for all $\la\ge 1$}\label{tail estimate0}
\end{align}
and
\begin{align}
\lim_{\la\to\infty} \left(\frac{\la}{2\pi}\right)^{N/2}Z_{\la}
=\det\left(D^2E(0)\right)^{-1/2}.\label{laplace}
\end{align}
The estimate (\ref{laplace}) can be proved by Laplace's method.
From now on,  we always assume $\la\ge 1$.
The log-Sobolev inequality (\ref{LSI RN}) implies that
for any bounded measurable function $V$, it holds that
\begin{align}
 \E^{\la}(F,F)+\int_{\RR^N}V(x)F(x)^2d\nu^{\la}(x)
&\ge 
-\frac{\la}{C}\log\left(\int_{\RR^N}e^{-\frac{C}{\la}V}d\nu^{\la}\right)
\|F\|_{L^2(\nu^{\la})}^2,\label{GNS0}
\end{align}
 where the constant $C$ is the same number as in
(\ref{LSI RN}).
We refer the reader to \cite{gross} for this estimate.
Let $\chi_0$ be a smooth function
with $\chi_0(u)=1$ for $|u|\le 1$
and $\chi_0(u)=0$ for $|u|\ge 2$.
Let $\kappa>0$ be a small number and set
$\chi_{0,\kappa}(x)=\chi_0(\kappa^{-1}|x|)$
and
$\chi_{1,\kappa}(x)=\sqrt{1-\chi_{0,\kappa}^2(x)}$.
Let $F\in {\rm D}(\E^{\la})$ and assume
$\|F\|_{L^2(\nu^{\la})}=1$ and
$\int_{\RR^N}F(x)d\nu^{\la}(x)=0$.
By an elementary calculation, 
\begin{align}
 \E^{\la}(F,F)&=\E^{\la}(F\chi_{0,\kappa},F\chi_{0,\kappa})
+\E^{\la}(F\chi_{1,\kappa},F\chi_{1,\kappa})\nonumber\\
&\quad
 -\int_{\RR^N}\left(|D\chi_{0,\kappa}|^2+|D\chi_{1,\kappa}|^2\right)F(x)^2
d\nu^{\la}(x).\label{finite dim 1}
\end{align}
This identity is called the IMS localization formula
(\cite{simon}).
We have
$|D\chi_{0,\kappa}|^2+|D\chi_{1,\kappa}|^2\le C'\kappa^{-2}$.
By applying (\ref{GNS0}),
\begin{align}
 \E^{\la}(F\chi_{1,\kappa},F\chi_{1,\kappa})&=
\E^{\la}(F\chi_{1,\kappa},F\chi_{1,\kappa})-\int_{\RR^N}
\delta\la^2(F\chi_{1,\kappa})^21_{|x|\ge \kappa}d\nu^{\la}\nonumber\\
&\quad+
\int_{\RR^N}
\delta\la^2(F\chi_{1,\kappa})^21_{|x|\ge \kappa}d\nu^{\la}\nonumber\\
&\ge -\frac{\la}{C}\log\left(\int_{\RR^N}
e^{\delta C\la 1_{|x|\ge \kappa}}d\nu^{\la}
\right)\|F\chi_{1,\kappa}\|_{L^2(\nu^{\la})}^2
+\delta\la^2\|F\chi_{1,\kappa}\|_{L^2(\nu^{\la})}^2\nonumber\\
&\ge \left\{-\frac{\la}{C}
\log\left(1+K_{\kappa}e^{\delta C\la-M_{\kappa}\la}\right)
+\delta\la^2
\right\}\|F\chi_{1,\kappa}\|_{L^2(\nu^{\la})}^2\nonumber\\
&\ge \left\{-\frac{\la}{C}K_{\kappa}e^{(\delta C-M_{\kappa})\la}
+\delta\la^2\right\}\|F\chi_{1,\kappa}\|_{L^2(\nu^{\la})}^2,
\end{align}
where we have used 
(\ref{tail estimate0}).
Thus, by choosing $\delta$ so that $\delta C<M_{\kappa}$, there exists
 $\delta'>0$
such that for large $\la$,
\begin{align}
 \E^{\la}(F\chi_{1,\kappa},F\chi_{1,\kappa})
&\ge \delta'\la^2\|F\chi_{1,\kappa}\|_{L^2(\nu^{\la})}^2.
\label{finite dim 2}
\end{align}
We estimate $\E^{\la}(F\chi_{0,\kappa},F\chi_{0,\kappa})$.
Note that the support of $F\chi_{0,\kappa}$ is
included in $\{x~|~|x|\le 2\kappa\}$.
Let $V=\{x~|~|x|< 3\kappa\}$.
For small $\kappa$,
by the Morse lemma, there exists an open neighborhood of $0$
$U$ and a
$C^{\infty}$-diffeomorphism
$\Phi$ : $y(\in U)\mapsto x(\in V)$ such that $\Phi(0)=0$ and
$E\left(\Phi(y)\right)=\frac{1}{2}|y|^2$ for all $y\in U$.
We write 
$m^{\la}(dy)=\left(\frac{\la}{2\pi}\right)^{N/2}e^{-\la|y|^2/2}dy$.
By using this coordinate, we have
\begin{align}
 \E^{\la}(F\chi_{0,\kappa},F\chi_{0,\kappa})&=
\int_V|D(F\chi_{0,\kappa})(x)|^2e^{-\la E(x)}Z_{\la}^{-1}dx\nonumber\\
&=\int_U|D(F\chi_{0,\kappa})\left(\Phi(y)\right)|^2
e^{-\frac{\la}{2}|y|^2}Z_{\la}^{-1}|\det(D\Phi(y))|dy\nonumber\\
&=\int_U
|\left\{(D\Phi(y))^{\ast}\right\}^{-1}
D\left\{\left(F\chi_{0,\kappa}\right)(\Phi(y))\right\}|^2
e^{-\frac{\la}{2}|y|^2}Z_{\la}^{-1}|\det(D\Phi(y))|dy.
\end{align}
We may assume that the mappings $y\mapsto \left\{(D\Phi(y))^{\ast}\right\}^{-1}$
and $y\mapsto |\det(D\Phi(y))|$
are Lipschitz continuous.
Let
 $\tilde{\sigma}_1$
be the smallest eigenvalue of
$(D\Phi(0))^{-1}\{(D\Phi(0))^{\ast}\}^{-1}$.
Then
there exists a positive function $\ep(\kappa)$ satisfying
$\lim_{\kappa\to 0}\ep(\kappa)=0$ such that
\begin{align}
\E^{\la}(F\chi_{0,\kappa},F\chi_{0,\kappa})
&\ge
(1-\ep(\kappa))\tilde{\sigma}_1
|\det
 D\Phi(0)|Z_{\la}^{-1}\left(\frac{\la}{2\pi}\right)^{-N/2}
\int_U|D\left\{\left(F\chi_{0,\kappa}\right)(\Phi(y))\right\}|^2
dm^{\la}(y)\nonumber\\
&\ge
(1-\ep(\kappa))
\tilde{\sigma}_1
|\det
 D\Phi(0)|Z_{\la}^{-1}\left(\frac{\la}{2\pi}\right)^{-N/2}\nonumber\\
&\quad \times
\la\left\{
\int_{\RR^N}(F\chio)^2(\Phi(y))dm^{\la}(y)
-\left(\int_{\RR^N}(F\chio)(\Phi(y))dm^{\la}(y)\right)^2
\right\},
\end{align}
where we have used the spectral gap
of the generator of the Dirichlet form
$\int_{\RR^N}|DF(y)|^2dm^{\la}(y)$ is $\la$.
We have
\begin{align}
\lefteqn{|\det
 D\Phi(0)|Z_{\la}^{-1}\left(\frac{\la}{2\pi}\right)^{-N/2}
\int_{\RR^N} 
(F\chio)^2(\Phi(y))dm^{\la}(y)}\nonumber\\
&=
|\det
 D\Phi(0)|Z_{\la}^{-1}\left(\frac{\la}{2\pi}\right)^{-N/2}
\int_U(F\chio)^2(\Phi(y))dm^{\la}(y)\nonumber\\
&=|\det
 D\Phi(0)|Z_{\la}^{-1}\left(\frac{\la}{2\pi}\right)^{-N/2}
\int_V(F\chio)^2(x)
e^{-\la E(x)}\left(\frac{\la}{2\pi}\right)^{N/2}
\left|\det(D(\Phi^{-1})(x))\right|dx\nonumber\\
&\ge(1-\ep(\kappa))
Z_{\la}^{-1}
\int_V(F\chio)^2(x)
e^{-\la E(x)}dx\nonumber\\
&=(1-\ep(\kappa))
Z_{\la}^{-1}
\int_{\RR^N}(F\chio)^2(x)
e^{-\la E(x)}dx
\end{align}
and
\begin{align}
&\int_{\RR^N}
\left(F\chio\right)(\Phi(y))dm^{\la}(y)\nonumber\\
&\quad=
\int_U(F\chio)(\Phi(y))dm^{\la}(y)\nonumber\\
&\quad=
\int_V(F\chio)(x)|\det(D(\Phi^{-1})(x))|
\left(\frac{\la}{2\pi}\right)^{N/2}Z_{\la}d\nu^{\la}(x)\nonumber\\
&\quad=
\int_V(F\chio)(x)
\left(|\det(D(\Phi^{-1})(x))|-|\det(D(\Phi^{-1})(0))|\right)
\left(\frac{\la}{2\pi}\right)^{N/2}Z_{\la}d\nu^{\la}(x)\nonumber\\
&\qquad\quad+|\det(D(\Phi^{-1})(0))|
\int_V(F\chio)(x)
\left(\frac{\la}{2\pi}\right)^{N/2}Z_{\la}d\nu^{\la}(x)\nonumber\\
&\quad=:I_1+I_2.
\end{align}
Here
\begin{align}
 |I_1|&\le
\ep(\kappa)\left(\frac{\la}{2\pi}\right)^{N/2}Z_{\la}
\|F\chio\|_{L^2(\nu^{\la})}
\end{align}
and by the Schwarz inequality,
\begin{align}
 &|I_2|\nonumber\\
&\le |\det(D(\Phi^{-1})(0))|
\left\{
\left|\int_{\RR^N}F(x)d\nu^{\la}(x)\right|+
\left|\int_{\RR^N}
F(y)\left(\chio(x)-1\right)d\nu^{\la}(x)\right|
\right\}\left(\frac{\la}{2\pi}\right)^{N/2}Z_{\la}
\nonumber\\
&\le |\det(D(\Phi^{-1})(0))|\,
\nu^{\la}\left(|x|\ge
 \kappa\right)^{1/2}\left(\frac{\la}{2\pi}\right)^{N/2}Z_{\la}\nonumber\\
&\le |\det(D(\Phi^{-1})(0))|\,
\sqrt{K_{\kappa}}e^{-\la M_{\kappa}/2}\left(\frac{\la}{2\pi}\right)^{N/2}Z_{\la}.
\end{align}
By the definition of $\Phi$, we have
$D^2E(0)=\{(D\Phi(0))^{\ast}\}^{-1}(D\Phi(0))^{-1}$.
Since the set of eigenvalues of
$(D\Phi(0))^{-1}\{(D\Phi(0))^{\ast}\}^{-1}$
and $\{(D\Phi(0))^{\ast}\}^{-1}(D\Phi(0))^{-1}$
are the same, we obtain $\tilde{\sigma}_1=\sigma_1$.
Thus, we get
\begin{align}
 \lefteqn{\E^{\la}(F\chio,F\chio)}\nonumber\\
&\ge
\la(1-\ep(\kappa))\sigma_1\|F\chio\|_{L^2(\nu^{\la})}^2\nonumber\\
&\quad-\la Z_{\la}^2\left(\frac{\la}{2\pi}\right)^{N}
\left\{
\ep(\kappa)
\|F\chio\|_{L^2(\nu^{\la})}+
|\det(D(\Phi^{-1})(0))|\,
\sqrt{K_{\kappa}}e^{-\la M_{\kappa}/2}
\right\}^2.
\label{finite dim 3}
\end{align}
By (\ref{finite dim 1}), (\ref{finite dim 2}), (\ref{finite dim 3})
and $\chi_{0,\kappa}^2(x)+\chi_{1,\kappa}^2(x)=1$ for all
$x$, we complete the proof of the lower bound.
The upper bound $\limsup_{\la\to\infty}\frac{e^{\la}_2}{\la}\le
\sigma_1$ can be proved by a standard way.
Let $v$ be a unit eigenvector such that $D^2E(0)v=\sigma_1 v$.
For this $v$, let
$F^{\la}(x)=\sqrt{\la \sigma_1} (x,v)$.
Then we have
 $\lim_{\la\to\infty}\frac{\E^{\la}(F^{\la},F^{\la})}{\la}=\sigma_1$,
$\lim_{\la\to\infty}\int_{\RR^N}F^{\la}(x)d\nu^{\la}(x)=0$
and $\lim_{\la\to\infty}\|F^{\la}\|_{L^2(\nu^{\la})}=1$
which imply the upper bound.
\end{proof}

\begin{rem}\label{first remark}
 {\rm
(1)~In the estimate of
$\E^{\la}(F\chio,F\chio)$,
we reduce the problem to Gaussian case
with the help of the Morse lemma.
The It\^o map is a measure preserving map between
$P_{x_0}(M)$ with the Brownian motion measure and
the Wiener space.
However, the derivative of the
It\^o map is not a bounded linear operator
between two tangent spaces
(\cite{driver1, cruzeiro-malliavin, elworthy-li1}).
In the study of the asymptotic behavior of the lowest eigenvalue
of a Schr\"odinger operator on $P_{x_0}(M)$ in \cite{aida-semiclassical},
the author reduced the local analysis 
to the analysis in Wiener spaces by using the It\^o map
and a ground state transformation.
At the moment, it is not clear that similar consideration
can be applied to the local analysis in the present problem.
In this paper, instead, we use the COH formula
in Lemma~\ref{coh formula}.

\noindent
(2)
~Let us consider a Dirichlet form
\begin{align}
\E^{A,\la}(F,F)&=\int_{\RR^N}|A(x)DF(x)|^2d\nu^{\la}(x),
\end{align}
where $A(x)$ is an $N\times N$ regular matrix-valued continuous
mapping 
on $\RR^N$ satisfying that
there exists a positive number $C>1$ such that
$C^{-1}|\xi|^2\le (A(x)\xi,\xi)\le C|\xi|^2$
for all $x,\xi$.
Suppose $\E^{A,\la}$ satisfies the above assumption 
(3) and (4).
Then, for the asymptotic behavior of the spectral gap of
$\E^{A,\la}$, the same result as in Theorem~\ref{RN} holds
replacing $\sigma_1$ by the lowest eigenvalue of
the Hessian of $E$ with respect to the Riemannian metric
defined by $g_{A}(x)(\xi,\xi)=|A(x)^{-1}\xi|^2$.
In that proof, we use the continuity of the map
$x\mapsto A(x)$.
In the case of $P_{x_0,y_0}(M)$, 
a local Poincar\'e inequality (\ref{poincare1})
and a log-Sobolev inequality (\ref{LSI loop}) holds.
However the mapping
$\gamma\mapsto A(\gamma)_{\la}$
is not a continuous mapping in the uniform convergence topology
and just a continuous
mapping in the topology of rough paths.
In this sense, we need the result in rough paths.
Moreover, in that case, the operator norm of
$A(\gamma)_{\la}$ is not uniformly bounded in $\gamma$.
Hence the argument is not so simple as in the above case.
Note that $A(\gamma)_{\la}$ depends on $\la$.
Hence we need to estimate $A(\gamma)_{\la}$ for large
$\la$.
In this calculation, we use the short time behavior of
the Hessian of the logarithm of the heat kernel.
}
\end{rem}
\section{Preliminary and Statement of results}\label{statement}

Let $(M,g)$ be an $n$-dimensional complete 
Riemannian manifold.
Let $d(x,y)$ denote the Riemannian distance between
$x$ and $y$.
Let $p(t,x,y)$
be the heat kernel of the diffusion semigroup
$e^{t\Delta/2}$ defined by the Laplace-Bertlami operator $\Delta$.
We refer the readers to \cite{hsu, ikeda-watanabe}
for stochastic analysis on manifolds.
The following assumption is natural for analysis on
Riemannian manifolds.

\begin{assumption}\label{assumption A}
\noindent
$(1)$
~There exist positive constants $C, C'$
such that for any $0<t\le 1$, $x,y\in M$,
\begin{align}
p(t,x,y)&\le Ct^{-n/2}e^{-C'd(x,y)^2/t}.
\end{align}

\noindent
$(2)$ The Ricci curvature of $M$ is bounded,
{\it i.e.}, $\|{\rm Ric}\|_{\infty}<\infty$.
\end{assumption}


The condition (2) implies $\int_Mp(t,x,y)dy=1$ holds for all
$t>0$ and $x\in M$, 
where $dy$ denotes the Riemannian volume.
In second main theorem (Theorem~\ref{main theorem 2}), 
we consider rotationally symmetric
Riemannian metrics.
We prove the above assumption holds true in such a case
by using the following observation in Lemma~\ref{gong-ma}.
Assumption~\ref{assumption A} (1) holds true
if the Ricci curvature is bounded from below and
the volume of small balls have uniform lower bound (\cite{li-yau}).
That is,
there exist $C>0$ and $l_0>0$ such that
${\rm vol}(B_l(x))\ge C l^{n}$ for all $0<l<l_0$ and any 
$x\in M$.
Here ${\rm vol}(B_l(x))$ denotes the volume
of the open metric ball $B_l(x)$ centered at $x$ with radius $l$.

In order to define (pinned) Brownian motion measure,
we assume $M$ satisfies Assumption~\ref{assumption A}.
Let $x_0\in M$.
The probability measure $\nu^{\la}_{x_0}$
on $P_{x_0}(M)$ satisfying the following is called 
the Brownian motion measure starting at $x_0$:

\noindent
For any Borel measurable subsets $A_k\subset M$~$(1\le k\le m)$
and
$0=t_0<t_1<\cdots<t_{m}\le 1$,
\begin{align}
&\nu^{\la}_{x_0}
\left(\{\gamma~|~
\gamma(t_1)\in A_1,\ldots,\gamma(t_m)\in A_m\}\right)
\nonumber\\
&\quad =\int_{M^m}
\prod_{k=1}^{m}p\left((t_{k}-t_{k-1})/\la,x_{k-1},x_k\right)
1_{A_k}(x_k)dx_1\cdots
 dx_{m}.
\end{align}
The process $\gamma(t)$ under $\nu^{\la}_{x_0}$ is a semimartingale.
When $M=\RR^n$, $\gamma(t)$ is the ordinary Brownian 
motion whose covariance matrix is
equal to
$tI/\la$.
Let $\pi : O(M)\to M$ be the orthonormal frame bundle
with the Levi-Civita connection.
We fix a frame $u_0=\{\ep_i\}_{i=1}^n\in \pi^{-1}(x_0)$.
By the mapping $u_0 :\RR^n\to T_{x_0}M$,
we identify $\RR^n$ with $T_{x_0}M$.
Let $\tau(\gamma)_t : T_{x_0}M\to T_{\gamma(t)}M$ 
denote the stochastic parallel translation along
$\gamma$.
For
smooth cylindrical function
$F(\gamma)=f(\gamma(t_1),\ldots,\gamma(t_m))
\in \FC(P_{x_0}(M))$ 
$(0<t_1<\cdots<t_m\le 1)$,
the $H$-derivative $DF(\gamma)$
is defined by
\begin{align}
 DF(\gamma)_t=
\sum_{i=1}^mu_0^{-1}\tau(\gamma)_{t_i}^{-1}(\nabla_i f)(\gamma(t_1),\ldots,\gamma(t_m))
t\wedge t_i,
\end{align}
where
$\nabla_i f$ denotes the derivative of $f$ with respect 
to the $i$-th variable.
Note that $DF(\gamma)\in \H:=H^1([0,1]\to \RR^n~|~h(0)=0)$.
Under Assumption~\ref{assumption A},
the symmetric form
\begin{align}
\E^{\la}(F,F) =
\int_{P_{x_0}(M)}|DF(\gamma)|_{\H}^2d\nu^{\la}_{x_0}(\gamma),
\qquad F\in \FC(P_{x_0}(M))
\end{align}
is closable.
We refer the reader to \cite{driver0, hsu, hsu1} for the closability.
The Dirichlet form of the smallest closed extension is denoted by the same
notation and the 
the generator $-L_{\la}$ is
a natural generalization
of OU operators in Gaussian cases.

We now consider the pinned case.
It is elementary fact that regular conditional probability
(pinned Brownian motion measure)
$\nu^{\la}_{x_0,y}(\cdot)=\nu^{\la}_{x_0}(\cdot~|~\gamma(1)=y)$ 
exists on $P_{x_0,y}(M)$ for 
$p(1,x_0,y)dy$-almost all
$y$.
However, it is necessary for us to define
$\nu^{\la}_{x_0,y}$ for all
$y\in M$ .
Actually, under Assumption~\ref{assumption A} (1) and (2), one can prove
that 
the regular conditional probability
$\nu^{\la}_{x_0,y}$
on $P_{x_0,y}(M)$
exists for all $y\in M$.
This can be checked by using the volume comparison theorem and 
the Kolmogorov criterion
(see \cite{aida-coh, hsu, driver01}).
Moreover, the pinned Brownian motion measure is equivalent to
the Brownian motion measure up to any time $t<1$ with respect to
the natural $\sigma$-field generated by the paths.
This implies that the pinned Brownian motion is a semimartingale
for $t<1$.
Hence the stochatsic parallel translation is well defined
and one can define the $H$-derivative
of a smooth $F(\gamma)=f(\gamma(t_1),\ldots,\gamma(t_m))\in
\FC(P_{x_0,y_0}(M))$ ($t_m<1$) by
$
 D_0F(\gamma)_t=
P_0\left(DF(\gamma)\right)_t,
$
where $P_0$ is the projection operator on
$\H$ onto
the subspace $\Ho:=\{h\in {\rm H}~|~h(1)=0\}$.
Using $D_0$ on $\FC(P_{x_0,y_0(M)})$, we can define
a symmetric bilinear form $\E^{\la}$
similarly to non-pinned case.
However, we need additional assumption on the Riemannian manifold
$M$ to prove the closability since $M$ may be non-compact.
Hence we consider the following assumption.

\begin{assumption}\label{assumption B}
$(\E^{\la},\FC(P_{x_0,y_0}(M)))$ is closable.
 \end{assumption}

We explain the reason why we need additional assumption.
Let
$b(t)=\int_0^tu_0^{-1}\tau(\gamma)_s^{-1}\circ d\gamma(s)$ 
$(0\le t\le 1)$, where $\circ d$ means Stratonovich integral.
The process $b(t)$ is anti-stochastic development of $\gamma(t)$.
Under the law $\nu^{\la}_{x_0}$, $b(t)$ is the ordinary Brownian motion
with variance $1/\la$.
We will discuss $b(t)$ later again in the explanation of the COH formula.
Note that the law of $\{b(t)\}_{0\le t\le 1}$ is singular with respect 
to the Brownian motion measure
under $\nu^{\la}_{x_0,y_0}$.
This is related to the singularity of the pinned Brownian motion itself.
The closability of $\E^{\la}$ can be proved by using the integration
by parts(=IBP) formula for $D, D_0$.
The formula contains
stochastic integrals with respect to
$b(t)$ and the integrability of the stochastic integrals
when $t$ converges to $1$ is the main issue to establish
the formula for the pinned measure.
See \cite{driver01, aida-coh, enchev-stroock1, enchev-stroock2, hsu, gordina}
for this problem.
If either
(i) $M$ is compact, or
(ii) $M$ is diffeomorphic to $\RR^n$ and the metric is flat outside a
      certain bounded set,
holds, by applying the Malliavin's quasi-sure analysis,
we can prove the integrability of the stochastic integrals 
and we obtain the IBP formula and the closability.
Also, under the condition,
\begin{align}
& \mbox{There exists a positive constant $C$ such that for any $0<t\le 1$ and
$z\in M$},\nonumber\\
&
|\nabla_z\log p(t,y_0,z)|\le C\frac{d(y_0,z)}{t}+\frac{C}{\sqrt{t}},
\end{align}
the IBP formula and the closability hold.
This inequality holds for any compact Riemannian manifolds
(\cite{hsu}). 
For rotationally symmetric Riemannian manifolds,
we will give a sufficient condition for this.
See Assumption~\ref{assumption C} and
Lemma~\ref{gong-ma} (2).

We now define 
a Dirichlet Laplacian
on a certain domain ${\cal D}$ 
in $P_{x_0,y_0}(M)$.

\begin{dfi}
Let $l$ be a positive number with $l>d(x_0,y_0)$.
let $B_l(y_0)$ denote the open ball centered at $y_0$
with radius $l$.
Define
\begin{align}
 {\cal D}_l&=\{\gamma\in P_{x_0,y_0}(M)~|~
\gamma(t)\in B_l(y_0)~~\mbox{for all $0\le t\le1$}
\}.
\end{align}
For $l=+\infty$, we set
${\cal D}_{\infty}=P_{x_0,y_0}(M)$.
 \end{dfi}

We may omit the subscript $l$ for simplicity.
In order to define the $H^1$-Sobolev spaces,
we assume Assumption~\ref{assumption B} for the moment.
Let $H^{1,2}(P_{x_0,y_0}(M),\nu^{\la}_{x_0,y_0})$ denote the $H^1$-Sobolev space
which is the closure of $\FC(P_{x_0,y_0}(M))$ with respect to
the norm
$\|F\|_{H^1}^2=\|F\|^2_{L^2(\nu^{\la}_{x_0,y_0})}+\E^{\la}(F,F)$.
Let
\begin{align}
H^{1,2}_0({\cal D},\nu^{\la}_{x_0,y_0})=
\left\{F\in H^{1,2}(P_{x_0,y_0}(M),\nu^{\la}_{x_0,y_0})~|~
\mbox{$F=0$~$\nu^{\la}_{x_0,y_0}$-a.s. outside ${\cal D}$}
\right\}
\end{align}
which is a closed linear subspace of
$H^{1,2}(P_{x_0,y_0}(M),\nu^{\la}_{x_0,y_0})$.

The non-positive generator $L_{\la}$ corresponding to the 
densely defined closed form 
$$
\E^{\la}(F,F),~~F\in H^{1,2}_0({\mathcal D},{\nu}^{\la}_{x_0,y_0})
$$
in the Hilbert space
$L^2({\mathcal D},{\nu}^{\la}_{x_0,y_0})$
is the Dirichlet Laplacian on 
${\mathcal D}$.
Let 
\begin{align}
e^{\la}_{Dir,1,{\mathcal D}}&=
\inf_{F(\ne 0)\in H^{1,2}_0\left({\mathcal D}\right)}
\frac{\int_{{\mathcal D}}|D_0F|^2d{\nu}^{\la}_{x_0,y_0}}
{\|F\|_{L^2(\nu^{\la}_{x_0,y_0})}^2}.
\end{align}
This is equal to $\inf\sigma(-L_{\la})$, where
$\sigma(-L_{\la})$ denotes the spectral set of $-L_{\la}$.
We next introduce 
\begin{align}
&\dirichlet\nonumber\\
&=
\sup_{G(\ne 0)\in L^2(\nu^{\la}_{x_0,y_0})}
\inf\Biggl\{
\frac{\int_{{\mathcal D}}|D_0F|^2d{\nu}^{\la}_{x_0,y_0}}
{\|F\|_{L^2({\nu}^{\la}_{x_0,y_0})}^2}~\Bigg |~
F\in H^{1,2}_0\left({\mathcal D}\right), 
\quad (F,G)_{L^2({\nu}^{\la}_{x_0,y_0})}=0\Biggr\}.
\end{align}
This is the generalized second lowest eigenvalue of
$-L_{\la}$.
When $l=+\infty$, 
$\ldirichlet=0$ and
$\dirichlet$ is equal to the spectral gap
of $-L_{\la}$ on the whole space
$P_{x_0,y_0}(M)$.
We use the notations
$e^{\la}_{1}$ and $e^{\la}_2$
instead of $\ldirichlet$ and $\dirichlet$
respectively in this case.
To state our first main theorem,
let us define the energy of $H^1$ path $\gamma$ belonging
to
$P_{x_0,y_0}(M)$,
\begin{align}
E(\gamma)=\frac{1}{2}\int_0^1|\gamma'(t)|^2_{T_{\gamma(t)}M}\,dt.
\label{energy function}
\end{align}
We use the same notation $D_0$ for the derivative
of the smooth function on the Hilbert manifold of the $H^1$ subset of
$P_{x_0,y_0}(M)$.
Note that $D^2_0E(\cxy)$ is a symmetric bounded linear operator
on $\Ho$.
See Lemma~\ref{S and T1} for the explicit form.
The following is our first main theorem.

\begin{thm}\label{main theorem 1}
Assume $M$ satisfies Assumptions~{\rm \ref{assumption A}},
{\rm \ref{assumption B}}.
Let $0<l<\infty$.
Assume that $l$ satisfies the
following.

\noindent
$(1)$ $l$ is smaller than the injectivity radius
at $y_0$.
In particular, 
there are no intersection of the closure of
$B_l(y_0)$ and ${\rm Cut}(y_0)$,
where ${\rm Cut}(y_0)$ denotes the cut-locus
of $y_0$.

\noindent
$(2)$
The Hessian of $k(z)=\frac{1}{2}d(z,y_0)^2$
satisfies that
$
\inf_{z\in B_{l}(y_0)}\nabla^2k(z)>1/2.
$

Then we have
\begin{align}
\lim_{\la\to\infty}
\frac{\dirichlet}{\la}=\sigma_1,\label{main theorem 1 identity}
\end{align}
where
$\sigma_1=\inf\sigma((D_0^2 E)(\cxy))$.
\end{thm}

Since $\nabla_z^2k(z)|_{z=y_0}=I_{T_{y_0}M}$,
the above conditions (1), (2) hold true for small
$l$.
Also, if $M$ is negatively curved manifold,
the condition (2) holds for all $l$.
We need condition (2)
to prove 
a COH formula by applying Lemma 3.1 in \cite{aida-coh}
although this may be just a technical condition.
Under the above condition, clearly 
the minimal geodesic $\cxy=\cxy(t)$~
$(0\le t\le 1)$~$(\cxy(0)=x_0, \cxy(1)=y_0)$
belongs to ${\mathcal D}$.
Further, 
$\lim_{\la\to\infty}\nu^{\la}_{x_0,y_0}({\cal D})=1$
holds true by a large deviation result
(see Section 5).

For a certain class of Riemannian manifolds $M$,
the same result holds for $P_{x_0,y_0}(M)$.
It is the second main theorem.
Let $M$ be a Riemannain manifold with a pole $y_0$.
That is, the exponential map
$\exp_{y_0} : T_{y_0}M\to M$ is a diffeomorphism.
We pick an
orthonormal frame $\tilde{u}_0$ of
$T_{y_0}M$.
Let $S^{n-1}$ be the unit sphere centered at the origin
in $\RR^n$.
We identify $\RR^n\setminus \{0\}$ with $(0,+\infty)\times S^{n-1}$ by
$(r,\Theta)(\in (0,+\infty)\times S^{n-1})
\mapsto r\Theta\in (\RR^n\setminus \{0\})$.
Let us define $\Psi : (0,+\infty)\times S^{n-1}\to M$ 
by
$x=\Psi(r,\Theta)=\exp_{y_0}\left(\tilde{u}_0(r\Theta)\right)$.
Then $r=d(y_0,x)$ holds.
The Riemannian metric $g$ is called rotationally symmetric at $y_0$
if the pull back of $g$ by $\Psi$ can be expressed as
\begin{align}
\Psi^{\ast}g=
dr^2+f(r)^2d\Theta^2, \label{rs metric}
\end{align}
$d\Theta^2$ denotes the standard Riemannian metric on the sphere.
Note that if $g$ is a smooth Riemannian metric on 
$M$, $f(r)$ is a $C^{\infty}$ function on $[0,\infty)$
satisfying $f(0)=0$ and $f'(0)=1$.
We consider the following assumption on $f$.

\begin{assumption}\label{assumption C}
Let $\varphi(r)=\log \frac{f(r)}{r}$.
The function $\varphi$ satisfies the following.

\noindent
$(1)$~$\varphi$ is a $C^{\infty}$ function on $[0,\infty)$.
The $k$-th derivative $\varphi^{(k)}(r)$ is
bounded function on $[0,\infty)$ for all $1\le k\le 4$.

\noindent
$(2)$~
There exists a $C^{\infty}$ function $\phi$ on $[0,\infty)$ 
such that
$\varphi(r)=\phi(r^2)$.

\noindent
$(3)$~$\inf_{r>0}r\varphi'(r)>-\frac{1}{2}$.
\end{assumption}

By Lemma A.2 in \cite{chow},
it is easy to deduce that
for any smooth function $f$ on $[0,\infty)$
satisfying $f(0)=0, f'(0)=1$ and
Assumption~\ref{assumption C} (2),
the Riemannian metric $dr^2+f(r)^2d\Theta^2$ on 
$\RR^n\setminus \{0\}$ can be extended
to a smooth Riemannian metric on $\RR^n$.
The above condition on $\varphi$ appeared in \cite{aida-precise}.
In \cite{aida-precise}, we assume all derivatives
$\varphi^{(k)}$ are bounded.
However we see that it is enough to assume the boundedness for
$1\le k\le 4$ by checking the calculations there.
We give examples of $\varphi$ which satisfies the above assumption.

\begin{exm}{\rm
For the hyperbolic space with the sectional curvature
$K=-a$, 
$\varphi_a(r)=\log\frac{\sinh \sqrt{a}r}{\sqrt{a}r}$.
This 
satisfies Assumption~$\ref{assumption C}$.
Actually $\varphi_a'(r)\ge 0$ for all $r$.
Clearly, small perturbations of $\varphi_a(r)$
satisfy the assumption.
Also if $\varphi_i$~$(1\le i\le n)$ satisfy 
the assumption, then 
so do the function $\sum_{i=1}^np_i\varphi_i$
for any positive numbers
$\{p_i\}$ with $\sum_{i=1}^np_i=1$.}
\end{exm}

The function $f$ satisfies the Jacobi equation
$f''(r)+K(r)f(r)=0$, where $K$ is the radial curvature
function.
It is natural to put the assumptions on $K$ instead of $f$.
In fact, it is proved in \cite{sasamori}
that necessary all estimates for the validity of
our second main theorem (Theorem~\ref{main theorem 2})
hold true under some assumptions on $K$.
Further related work is in progress.

The quantity $r\varphi'(r)$ is related to
the second derivative of the squared distance function
as in the following lemma
(\cite{greene-wu, aida-precise}).

\begin{lem}\label{hessian of d2}
For $r=d(y_0,z)$,
we have
\begin{align}
 \nabla^2_z\left(\frac{r^2}{2}\right)
&=I_{T_zM}+r\varphi'(r)P_z^{\perp},
\end{align}
where 
$v_z\in T_zM$ is the element such that
$\exp_z(v_z)=y_0$ and
$P_z^{\perp}$ denotes the orthogonal projection
onto the orthogonal complement of the $1$ dimensional
subspace spanned by $v_z\in T_zM$.
\end{lem}

By this lemma, we see that Assumption~\ref{assumption C} (3)
implies the condition (2) in Theorem~\ref{main theorem 1}
with $l=+\infty$.

\begin{lem}\label{sufficient condition for assumption A}
Suppose $f$ satisfies Assumption~{\rm \ref{assumption C}} $(1), (2)$ 
and $\inf_{r>1}f(r)>0$.
Then 
Assumptions~{\rm \ref{assumption A}}, {\rm \ref{assumption
 B}}
hold.
\end{lem}

\begin{proof}
By Lemma 1.21 in \cite{chow} (see also Proposition 9.106 in
 \cite{besse}),
it is easy to see the boundedness of the Ricci curvature
under the Assumption~\ref{assumption C} (1), (2).
To prove the Gaussian upper bound in Assumption~\ref{assumption A} (1),
it suffices to prove that there exists $C>0$ such that
$\inf_{x\in M}{\rm vol}(B_l(x))\ge Cl^n$ for small $l>0$
because the Ricci curvature is bounded.
Also under the assumption $\|\varphi'\|_{\infty}<\infty$,
we obtain there exist positive constants $C(\ep,R)$ and
$c(\ep,R)$ for any $\ep>0$ and $R>0$ such that
\begin{align}
c(\ep,R)\le \frac{f(r')}{f(r)}\le
C(\ep,R)
\qquad \mbox{for any $r, r'$ with
$r,r'\ge R, |r-r'|\le \ep$}
\end{align}
and
$\lim_{\ep\to 0}c(\ep,R)=\lim_{\ep\to 0}C(\ep,R)=1$.
By using this and $\inf_{r\ge 1}f(r)>0$,
it is not difficult to show the uniform lower boundedness of the volume
by this estimate.
Assumption~\ref{assumption B} follows from 
the estimate of $\nabla_z\log p(t,y_0,z)$
in (\ref{loggrad}).
\end{proof}

Actually, (1.58) in \cite{chow} implies that
the sectional curvature is bounded
under Assumption~\ref{assumption C}.
Hence, we may use comparison theorem of heat kernels
to prove the Gaussian upper bound.
We refer the reader to \cite{hsu} for the comparison theorem.
Also we note that $\inf_{r>0} r\varphi'(r)>-1$ implies
$\inf_{r>1}f(r)>0$.

The following is
our second main theorem.
We prove the positivity of $e^{\la}_2$ in
more general setting in
Theorem~\ref{main theorem 3}.

\begin{thm}\label{main theorem 2}
Let $M$ be a rotationally symmetric Riemannian manifold
with a pole $y_0$.
Suppose $f$ in $(\ref{rs metric})$ satisfies
Assumption~{\rm \ref{assumption C}}.
Then $e_2^{\la}>0$ holds for all $\la>0$ and
\begin{align}
 \lim_{\la\to\infty}\frac{e_2^{\la}}{\la}=\sigma_1,\label{asymptotics of ela2}
\end{align}
where $\sigma_1$ is the same number as in
Theorem~$\ref{main theorem 1}$.
\end{thm}

We make remarks on Theorem~\ref{main theorem 1}
and Theorem~\ref{main theorem 2}.

\begin{rem}
{\rm 

\noindent
$(1)$
It is not clear whether the same result as in Theorem~\ref{main theorem 2}
holds or not for $P_{x_0,y}(M)$ ($y\ne y_0$)
under Assumption~\ref{assumption C}.
It is more interesting to study non-rotationally general cases.

\noindent
$(2)$~By checking the proof, the same results as
in Theorem~\ref{main theorem 2} hold if the following 
are satisfied,
\begin{itemize}
 \item[(i)] $d(x_0,y_0)$ is smaller than $l$ which satisfies 
Theorem~\ref{main theorem 1} (1) and (2),
\item[(ii)] the log-Sobolev inequality (\ref{LSI loop}) holds,
\item[(iii)] the tail estimate (\ref{tail estimate}) holds.
\end{itemize}

\noindent
$(3)$
If the sectional curvature along
the geodesic $\cxy$ is positive, then 
$\inf\sigma(D_0^2 E(\cxy))<1$ 
and the bottom of the spectrum is an eigenvalue
of $D_0^2E(\cxy)$ and 
is not an essential spectrum.
While the curvature is strictly negative, 
$\inf\sigma(D_0^2 E(\cxy))=1$
and $1$ is not an eigenvalue and belongs to essential spectrum.
This suggests that the second lowest eigenvalue, or more generally,
some low-lying spectrum of the OU operator (with Dirichlet boundary
condition)
on ${\mathcal D}$ or $P_{x_0,y_0}(M)$
over a positively curved manifold belongs to the discrete spectrum,
while the second lowest eigenvalue is embedded in the
essential spectrum in the case of negatively curved manifolds.
In fact, in the proof of upper bound in the main theorems,
we use ``approximate second eigenfunctions'' which
are defined by the eigenfunction which achieves the value
$\inf\sigma(D_0^2 E(\cxy))$ approximately.
If some isometry group acts on $M$ with the fixed points
$x_0$ and $y_0$, we may expect the discrete spectrum have some
multiplicities.
We show these kind of results in the case where $M$ is a compact
Lie group in a forthcoming paper.
}
\end{rem}

As mentioned in the Introduction,
the spectral gap $e_2^{\la}$ for $P_{x_0}(M)$
is defined similarly
and $e_2^{\la}>0$ for all $\la$.
This is due to Fang.
He established a
COH formula and proved the existence of the
spectral gap
in the case where $M$ is compact and $\la=1$.
However, it is obvious that
the same result holds true
on a complete Riemannian manifold with bounded
Ricci curvature for all $\la>0$ .
See also \cite{chl, gong-ma, aida-semiclassical0, aida-gradient}.
The variant of the
COH formula in the loop space case is important in our case also.
To explain the COH formula,
we need some preparations.
Let ${\mathfrak F}_t=\sigma\left(\gamma(s), 0\le s\le t\right)
\vee {\cal N}$,
where ${\cal N}$ is the set of all null sets 
with respect to $\nu^{\la}_{x_0}$.
Then $b(t)=\int_0^tu_0^{-1}\tau(\gamma)_s^{-1}\circ d\gamma(s)$
is an ${\mathfrak F}_t$-Brownian motion with the covariance
$E^{\nu^{\la}_{x_0}}[(b(t),u)(b(s),v)]=(u,v)\frac{t\wedge s}{\la}$
~$(u,v\in \RR^n)$ on $\RR^n$ under $\nu^{\la}_{x_0}$.
We simply say $b(t)$ is a Brownian motion with variance $1/\la$
in this paper.
We recall the notion of the trivialization.
Let $T\in \Gamma(TM\otimes T^{\ast}M))$ be a
$(1,1)$-tensor on $M$,
that is, $T$ is a linear transformation on
each tangent space.
We write
\begin{align}
 \overline{T(\gamma)}_t=u_0^{-1}\tau(\gamma)_t^{-1}T(\gamma(t))\tau(\gamma)_tu_0
\in L(\RR^n,\RR^n).
\end{align}
The definition for general $T\in \Gamma\left((\otimes^pTM)\otimes 
(\otimes^q T^{\ast}M)\right)$ is similar.

We now state COH formula on
$P_{x_0}(M)$.
Below, we use the notation
$\Ltwo:=L^2([0,1]\to \RR^n, dt)$.

\begin{lem}\label{fang COH formula}
Assume $\|{\rm Ric}\|_{\infty}<\infty$.
Let $F\in H^1(P_{x_0}(M),\nu^{\la}_{x_0})$.
Then
\begin{align}
 F(\gamma)-E^{\nu^{\la}_{x_0}}[F]&=
\int_0^1
\left(E\left[\left\{\left((I+R_{0,\la}(\gamma))^{-1}\right)^{\ast}
(DF)(\gamma)'\right\}_t~|~{\mathfrak F}_t\right],db(t)\right),
\end{align}
where $\left(R_{0,\la}(\gamma)\varphi\right)(t)=\frac{1}{2\la}
\overline{{\rm Ric}(\gamma)}_t\int_0^t\varphi(s)ds$,
$\ast$ indicates the adjoint operator
on $\Ltwo$ and
$DF(\gamma)_t'=\frac{d}{dt}DF(\gamma)_t$.
Also $I$ denotes the identity operator on $\Ltwo$.
\end{lem}

The second derivative of $\log p(t,x,y)$ is related 
to the COH formula on $P_{x_0,y_0}(M)$.
Under Assumption~\ref{assumption C},
we have a good estimate on the first and second derivatives
of $\log p(t,y_0,z)$ with respect to $z$.
Similar estimates of
the heat kernel hold in a
compact set outside cut-locus
when $M$ is a compact Riemannian manifold.
This is studied by Malliavin and Stroock~\cite{ms}
and Gong-Ma~\cite{gong-ma}.
Their results clearly can be extended to
non-compact $\RR^n$ with a nice Riemannian metric which
coincides with the Euclidean metric
outside a bounded set.
The estimates are as follows.

\begin{assumption}\label{assumption D}
 For any compact subset $F\subset {\rm Cut}(y_0)^c$ and
$0<t\le 1$ there exists $C_F>0$ such that
\begin{align}
\sup_{z\in F}
\left|t\nabla^2_z\log p(t,y_0,z)+
\nabla^2_z\left(\frac{1}{2}
d(y_0,z)^2\right)\right|\le C_Ft^{1/2}.
\label{loghessian 0}
\end{align}
\end{assumption}

The following (1) and (2) can be found in \cite{aida-precise}
and \cite{gong-ma}
respectively.

\begin{lem}\label{gong-ma}
$(1)$
Let $M$ be a compact Riemannian manifold
or $\RR^n$ with a Riemannian metric which
coincides with the Euclidean metric
outside a bounded set.
Then Assumption~{\rm \ref{assumption D}}
is satisfied.

\noindent
$(2)$~Suppose Assumption~{\rm \ref{assumption C}} $(1)$ and $(2)$.
Then Assumption~{\rm \ref{assumption D}} is satisfied.
Actually the following stronger
inequalities are valid:

Let $T>0$.
There exist positive constants $C_1, C_2$
which may depend on $T$ such that for all $0<t\le T$,
\begin{align}
&\sup_{z\in M}\left|t\nabla_z\log p(t,y_0,z)-v_z\right|
\le C_1t,\label{loggrad}\\
&\sup_{z\in M}\left|
t\nabla_z^2\log p(t,y_0,z)+
I_{T_zM}+d(y_0,z)\varphi'(d(y_0,z))P_z^{\perp}
\right|\le C_2t,\label{loghessian}
\end{align}
where $v_z$ and $P_z^{\perp}$
are defined in Lemma~$\ref{hessian of d2}$.
\end{lem}

The important point in the estimate (\ref{loghessian})
is that the norm of the second derivative of
$t\log p(t,y_0,z)$ is bounded from above by a linear function
of $d(y_0,z)$.
Probably, the estimates (\ref{loggrad}) and (\ref{loghessian})
hold under weaker assumptions on $\varphi$.
It is natural and interesting to study 
non-rotationally symmetric general cases.

Our Dirichlet Laplacian is defined
on the set of paths which are restricted in
the small ball.
Therefore, even if we vary the Riemannian metric outside the
ball, the spectral property of the operator
would not change.
We explain this reasoning more precisely.
Let $(M,g)$ and $(M',g')$ be Riemannian manifolds satisfying
Assumption~\ref{assumption B}.
Let $y_0\in M, y_0'\in M'$ and
$B_l(y_0)\subset M, B_l(y_0')\subset M'$ be open metric
balls. Let $x_0\in B_l(y_0)$.
Let $l_{\ast}>l$.
Assume that $l_{\ast}$ is smaller than the injectivity radius at
$y_0$.
We assume that there exists a 
Riemannian isometry $\Phi : B_{l_{\ast}}(y_0)\to B_{l_{\ast}}(y_0')$.
Then $\Phi(B_l(y_0))=B_l(y_0')$.
Let $x_0'=\Phi(x_0)$.
Let $\nu^{\la}_{M,x_0,y_0}$ and $\nu^{\la}_{M',x_0',y_0'}$
denote the pinned measures on each manifold.
We write
\begin{align*}
{\cal D}&=\{\gamma\in P_{x_0,y_0}(M)~|~\gamma(t)\in B_l(y_0)
~\mbox{for all $0\le t\le 1$}\},\\
{\cal D}'&=\{\gamma\in P_{x_0',y_0'}(M')~|~\gamma(t)\in B_l(y_0')
~\mbox{for all $0\le t\le 1$}\}.
\end{align*}
Let $A\subset {\cal D}$ be a Borel measurable subset.
Define $\Phi : {\cal D}\to {\cal D}'$ by
$\Phi(\gamma)(t)=\Phi(\gamma(t))$.
$p^M(t,x,y)$ and $p^{M'}(t,x',y')$ denote the heat kernels on
$M$ and $M'$.
Note that $p^M(t,x,y)\ne p^{M'}(t,\Phi(x),\Phi(y))$ $x,y\in
B_{l_{\ast}}(y_0)$
generally.
However, by the uniqueness of the solution of
stochastic differential equations,
we have
\begin{align}
\frac{\nu^{\la}_{M,x_0,y_0}(A)}{\nu^{\la}_{M,x_0,y_0}({\cal D})}
&=
\frac{\nu^{\la}_{M',x_0',y_0'}(\Phi(A))}{\nu^{\la}_{M',x_0',y_0'}({\cal D}')}.
\end{align}
By this, for any bounded Borel measurable function
$F$ on ${\cal D}'$,
\begin{align}
 \int_{{\cal D}}F(\Phi(\gamma))
\frac{d\nu^{\la}_{M,x_0,y_0}(\gamma)}{\nu^{\la}_{M,x_0,y_0}({\cal
 D})}
&=
\int_{{\cal D}'}F(\gamma)
\frac{d\nu^{\la}_{M',x_0',y_0'}(\gamma)}
{\nu^{\la}_{M',x_0',y_0'}({\cal D}')}.
\end{align}
Let $F\in H^{1,2}(P_{x'_0,y'_0}(M'))$.
If $F\in H^{1,2}_0({\cal D}',\nu^{\la}_{x_0',y_0'})$, then
$$
\tilde{F}(\gamma):=
F\left(\Phi(\gamma)\right) \chi\left(\sup_{0\le t\le 1}
d'(\Phi(\gamma)(t),y_0')\right)
\in H^{1,2}_0({\cal D},\nu^{\la}_{x_0,y_0}),
$$
where $\chi=\chi(t)$ is a non-negative smooth function such that
$\chi(t)=1$ for $t\le \frac{l+l_{\ast}}{2}$ and
$\chi(t)=0$ for $t\ge \frac{l+2l_{\ast}}{3}$.
Moreover 
$\|D_0F\|_{L^2(\nu^{\la}_{M',x'_0,y'_0}/\nu^{\la}_{M',x_0',y_0'}({\cal D}'))}
=\|D_0\tilde{F}\|_{L^2(\nu^{\la}_{M,x_0,y_0}/\nu^{\la}_{M,x_0,y_0}({\cal
D}))}$.
To prove these results, we need 
$\sup_{0\le t\le 1}d(\gamma(t),\Phi^{-1}(y_0'))\in
H^{1,2}(P_{x_0,y_0}(M))$ 
which can be found in Lemma 2.2 and Remark 2.4 in
\cite{aida-coh}.

The above argument implies that
$$
e^{\la}_{Dir,2,{\cal D}}=e^{\la}_{Dir,2,{\cal D}'}.
$$
Hence, in the proof of Theorem~\ref{main theorem 1},
we may assume that 
$M$ is diffeomorphic to ${\mathbb R}^n$
and the Riemannian metric is flat outside a certain bounded
subset and Assumption~\ref{assumption D} is satisfied.
The key ingredient of the proof of Theorem~\ref{main theorem 1}
is a version of the COH formula in
\cite{aida-coh2} which can be extended to the above non-compact
${\mathbb R}^n$ case with a nice Riemannian metric.
Since the COH formula is strongly related to the heat kernel $p(t,x,y)$ 
on $M$ itself, the above observation is important.
We explain COH formula on $P_{x_0,y_0}(M)$.
Let
$V_{y_0}^{\la}(t,z)=\grad_z\log p\left(\frac{1-t}{\la},y_0,z\right)$
~$(0\le t<1)$.
We write
$$
\overline{V_{y_0}^{\la}(t,\gamma)}_t=
u_0^{-1}\tau(\gamma)_t^{-1}V_{y_0}^{\la}(t,\gamma(t))
\in \RR^n.
$$
Also 
$\DVbar$
denotes an $n\times n$ matrix.
More explicitly,
\begin{equation}
\DVbar=
u_0^{-1}\tau(\gamma)_t^{-1}\nabla_z\grad_z \log p\left(\frac{1-t}{\la},y_0,z\right)
\Bigg |_{z=\gamma(t)}\tau(\gamma)_tu_0.
\end{equation}
Let 
$w(t)=b(t)-\frac{1}{\la}\int_0^t\overline{V_{y_0}^{\la}(s,\gamma)}_sds$.
This process is defined for $t<1$ and it is not difficult to check that
this can be extended continuously up to $t=1$.
Let $\mathcal{N}^{x_0,y_0,t}$ be the set of all null sets of
$\nu^{\la}_{x_0,y_0}|_{{\mathfrak F}_t}$ and set
${\mathfrak G}_t={\mathfrak F}_t\vee \mathcal{N}^{x_0,y_0,1}$.
Then $w$ is an ${\mathfrak G}_t$-adapted Brownian
motion for $0\le t\le 1$ such that
$E^{\nu^{\la}_{x_0,y_0}}
[\left(w(t),u\right)\left(w(s),v\right)]
=\frac{t\wedge s}{\la}(u,v)$ for any
$u,v\in \RR^n$.
Let
\begin{align}
K(\gamma)_{\la,t}=
-\frac{1}{2\la}\Ricbar
+\frac{1}{\la}\DVbar.
\end{align}
Let $M(\gamma)_{\la,t}$ be the linear mapping on $\RR^n$ 
satisfying the differential equation:
\begin{align}
M(\gamma)_{\la,t}'&=K(\gamma)_{\la,t}M(\gamma)_{\la,t}
\quad 0\le t<1,\\
M(\gamma)_{\la,0}&=I.
\end{align}
Using $M$ and $K$, we define for a bounded measurable function 
$\varphi$ with ${\rm supp}\, \varphi\subset [0,1)$,
\begin{align}
J(\gamma)_{\la}\varphi(t)&=
(M(\gamma)_{\la,t}^{\ast})^{-1}\int_t^1
M(\gamma)_{\la,s}^{\ast}K(\gamma)_{\la,s}\varphi(s)ds.
\end{align}
The operator $\left((I+R_{0,\la}(\gamma))^{-1}\right)^{\ast}$
in the COH formula in Lemma~\ref{fang COH formula}
coincides with $J(\gamma)_{\la}$ which is
obtained by setting 
$K(\gamma)_{\la}=-\frac{1}{2\la}\overline{{\rm Ric}(\gamma)_t}$
in the above.
Also let 
\begin{equation}
A(\gamma)_{\la}=I+J(\gamma)_{\la}.
\end{equation}
We are ready to state our COH formula for functions on
$P_{x_0,y_0}(M)$
and its immediate consequences.

\begin{lem}\label{coh formula}~
\noindent
$(1)$ Assume $M$ is diffeomorphic to $\RR^n$ and
the Riemannian metric is flat outside a bounded subset.
Let $0<l<\infty$.
Suppose ${\cal D}(={\cal D}_l)$ 
satisfies conditions $(1), (2)$ in Theorem~$\ref{main theorem 1}$.
Let $F\in H^{1,2}_0({\mathcal D})$.
\begin{enumerate}
\item[{\rm (i)}] It holds that
$D_0F(\gamma)=0$ for $\nu^{\la}_{x_0,y_0}$-almost all 
$\gamma\in {\mathcal D}^c$.
\item[{\rm (ii)}] 
There exists $\la_{\ast}>0$ such that 
$A(\gamma)_{\la}$ can be extended to a
bounded linear operator on $\Ltwo$ for each $\gamma$
for all $\la\ge\la_{\ast}$.
Let $a(\la)=\esssup\left\{\|A(\gamma)_{\la}\|_{op}^2~|~\gamma\in
{\cal D}\right\}$.
Here $\|\cdot\|_{op}$ denotes the operator norm.
$\displaystyle{\sup_{\la\ge \la_{\ast}}a(\la)<\infty}$ holds
and for $\la\ge \la_{\ast}$,
the following COH formula holds:
\begin{align}
E^{\nu^{\la}_{x_0,y_0}}[F|{\mathfrak G}_t]=
E^{\nu^{\la}_{x_0,y_0}}[F]+
\int_0^t\left(H(s,\gamma),dw(s)\right),\quad 0\le t\le 1,
\label{COH loop 1}
\end{align}
where
\begin{align}
H(s,\gamma)=
E^{\nu^{\la}_{x_0,y_0}}\left[
A(\gamma)_{\la}(D_0F(\gamma)')(s)| {\mathfrak G}_s\right].
\quad\label{COH loop 2}
\end{align}
and $D_0F(\gamma)'_t=\frac{d}{dt}(D_0F)(\gamma)_t$.
Moreover
the following inequalities hold for $\la\ge \la_{\ast}$.
\begin{align}
E^{{\nu}^{\la}_{x_0,y_0}}\left[
F^2\log \left(F^2/\|F\|^2_{L^2({\nu}_{x_0,y_0}^{\la})}\right)
\right]
&\le
\frac{2a(\la)^2}{\la}E^{{\nu}^{\la}_{x_0,y_0}}\left[
|D_0F|^2\right],\label{lsi1}\\
\la
E^{\nu^{\la}_{x_0,y_0}}\left[\left(F-E^{{\nu}^{\la}_{x_0,y_0}}[F]\right)^2
\right]
&\le E^{\nu^{\la}_{x_0,y_0}}\left[
|A(\gamma)_{\la}D_0F|^2\right].\label{poincare1}
\end{align}
\end{enumerate}

\noindent
$(2)$~Assume $M$ is a rotationally symmetric Riemannian manifold
with a pole $y_0$. Suppose Assumption~{\rm \ref{assumption C}}.
\begin{itemize}
\item[{\rm (i)}] The operator $A(\gamma)_{\la}$ can be extended to
a bounded linear operator on $\Ltwo$ for each $\gamma$ for all $\la>0$.
Moreover for each $\la_0>0$, there exists a positive constant $C_0$
which depends only on $\varphi$ and $\la_0$ such that
for all $\la\ge \la_0$,
\begin{align}
 \|A(\gamma)_{\la}\|_{op}&\le
C_0\rho_{y_0}(\gamma)\qquad \mbox{for any $\gamma$},
\end{align}
where $\rho_{y_0}(\gamma)=1+\max_{0\le t\le 1}d(y_0,\gamma(t))$.

 \item[{\rm (ii)}] For $F\in H^{1,2}(P_{x_0,y_0}(M),\nu^{\la}_{x_0,y_0})$,
the COH formula~$(\ref{COH loop 1}), (\ref{COH loop 2})$ hold.
\item[{\rm (iii)}] For each $\la_0>0$, there exists a positive constant $C_1$
which depends only on $\varphi$ and $\la_0$ such that
for any $\la\ge \la_0$ and $F\in \FC(P_{x_0,y_0}(M))$,
\begin{align}
&\int_{P_{x_0,y_0}(M)}
F^2(\gamma)\log\left(F(\gamma)^2/\|F\|_{L^2(\nu^{\la}_{x_0,y_0})}^2\right)
d\nu^{\la}_{x_0,y_0}(\gamma)\nonumber\\
&\le \int_{P_{x_0,y_0}(M)}
\frac{C_1}{\la}\rho_{y_0}(\gamma)^2
|D_0F(\gamma)|^2d\nu^{\la}_{x_0,y_0}(\gamma).\label{LSI loop}
\end{align}
\end{itemize}
\end{lem}

\begin{proof}
The proof of $(1)$ is similar to that in \cite{aida-coh2}.
$(2)$ follows from
Lemma 3.2 and Theorem 3.3 in \cite{aida-coh}
and Lemma 2.3 in \cite{aida-coh2}.
In the present case, we have
\begin{align}
K(\gamma)_{\la,t}=
\frac{1}{1-t}\left(-\alpha+C_1(t)\right)+C_2(t),
\end{align}
where 
\begin{align*}
\alpha&=1+\inf_{r>0}r\varphi'(r)>1/2,\\
 C_1(t)&=\left\{\left(\inf_{r>0}r\varphi'(r)\right)-d(y_0,\gamma(t))
\varphi'\Bigl(d(y_0,\gamma(t))\Bigr)\right\}
\overline{P^{\perp}(\gamma)}_{\gamma(t)}\\
&\qquad\quad +
\left(\inf_{r>0} r\varphi'(r)\right)
 \overline{P(\gamma)}_{\gamma(t)},\\
|C_2(t)|&\le \frac{C}{\la}
\end{align*}
and $C$ is a positive constant.
The case $\la=1$ is considered in \cite{aida-coh}
and the estimate in the hyperbolic space case with general $\la$
can be found in Remark 2.4 in \cite{aida-coh2}.
The proof of general cases are similar to them.
\end{proof}

Under the assumption in the lemma above,
$A(\gamma)_{\la}$ is a bounded linear operator on
$\Ltwo$ for $\nu^{\la}_{x_0,y_0}$ almost all $\gamma$.
However, we cannot expect the usual continuity property
of the mapping $\gamma\mapsto A(\gamma)_{\la}$ 
because they are defined by using It\^o's stochastic integrals.
The inequality (\ref{poincare1}) implies that 
$\liminf_{\la\to\infty}\frac{e^{\la}_{Dir,2,{\mathcal D}}}{\la}>0$.
On the other hand, we cannot conclude $e_2^{\la}>0$ for
$P_{x_0,y_0}(M)$ by the
log-Sobolev inequality (\ref{LSI loop}) because
the operator norm 
$A(\gamma)_{\la}$ is not uniformly bounded. 

As mentioned in the Introduction,
the same result as in Theorem~\ref{main theorem 2}
holds for $P_{x_0}(M)$.
We prove it
as a warm up before proving our main theorems.
For simplicity, we assume $M$ is compact.
After the proof, we explain different points of the proof
in the loop space case.

\begin{thm}\label{e2la on Px0}
Let $M$ be a compact Riemannian manifold.
Let $e_2^{\la}$ be the spectral gap
of the Dirichlet form $\E^{\la}$ 
on $P_{x_0}(M)$ with $\nu^{\la}_{x_0}$.
Then $e_2^{\la}>0$ for all $\la>0$
and
 \begin{align}
  \lim_{\la\to\infty}
\frac{e_2^{\la}}{\la}=1.\label{e2 asymptotics 2}
 \end{align}
\end{thm}

\begin{proof}
We use the COH formula 
(\ref{fang COH formula}).
By using 
\begin{align}
 \|(I+R_{0,\la}(\gamma))^{-1}\|_{op}\le 1+\frac{C}{\la}
\quad \mbox{for any $\la\ge \la_0>0$},\label{coefficient operator 1}
\end{align}
we get
\begin{align}
 E^{\nu^{\la}_{x_0}}
\left[(F-E^{\nu^{\la}_{x_0}}[F])^2\right]
&\le
\frac{1}{\la}\left(1+\frac{C}{\la}\right)^2E\left[|DF(\gamma)|^2\right].
\end{align}
Here $C$ depends on $\la_0$.
Since $e^{\la}_1=0$ and the corresponding eigenfunction is
a constant function,
we have $e_2^{\la}\ge \la(1+\frac{C}{\la})^{-2}$ which proves
that $\liminf_{\la\to\infty}\frac{e^{\la}_2}{\la}\ge 1$.
We prove converse estimate.
To this end, we consider a candidate of
approximate second (generalized) eigenfunction.
Let $\varphi\in \Ltwo$ and assume
$\|\varphi\|_{\Ltwo}=1$.
Let $F(\gamma)=\sqrt{\la}\int_0^1\left(\varphi(t),db(t)\right)$.
Then $E^{\nu^{\la}_{x_0}}[F]=0$ and
$E^{\nu^{\la}_{x_0}}[F^2]=1$.
We have $F\in {\rm D}(\E)$ and
\begin{align*}
 D_{h}
\int_0^1\left(\varphi(t),db(t)\right)&=
\int_0^1\left(\varphi(t),h'(t)\right)dt
+\int_0^1
\Big\langle\varphi(t),
\int_0^t\left(\overline{R(\gamma)}_s(h(s),\circ db(s))(\circ db(t))\right)
\Big\rangle,
\end{align*}
where $\overline{R(\gamma)}_t$ is the trivialization
of the Riemannian curvature tensor
and $\langle\cdot,\cdot\rangle$ also denotes the inner
product in $\RR^n$.
The readers are referred to
\cite{cruzeiro-malliavin, aida-irred} for this formula.
See \cite{driver0, leandre1} also.

Since
\begin{align}
&\int_0^1
\Big\langle \varphi(t),
\int_0^t\left(\overline{R(\gamma)}_s(h(s),\circ db(s))(\circ db(t))\right)
\Big\rangle\nonumber\\
&=\int_0^1\varphi^j(t)
\Big\langle
\ep_j,\int_0^t\left(\overline{R(\gamma)}_s(h(s),\circ db(s))(\circ db(t))\right)
\Big\rangle\nonumber\\
&=\int_0^1
\int_0^t\Big\langle \overline{R(\gamma)}_s(h(s),\circ db(s))(\ep_j),
\circ d\int_t^1\varphi^j(s)db(s)\Big\rangle\nonumber\\
&=\int_0^1
\Big\langle\overline{R(\gamma)}_t(\circ db(t),h(t))(\ep_j),
\int_t^1\varphi^j(s)db(s)\Big\rangle\nonumber\\
&=\int_0^1\Big\langle\overline{R(\gamma)}_t(\ep_j,\int_t^1\varphi^j(s)db(s))(\circ
 db(t)), h(t)
\Big\rangle\nonumber\\
&=\int_0^1\Big\langle\overline{R(\gamma)}_t(\int_t^1\varphi(s)db^i(s),\ep_i)
(\circ db(t)), h(t)\Big\rangle\nonumber\\
&=\int_0^1
\Big\langle\int_t^1\overline{R(\gamma)}_u\left(\int_u^1\varphi(s)db^i(s),
\ep_i\right)(\circ db(u)), h'(t)\Big\rangle dt,
\end{align}
we obtain
\begin{align}
 DF(\gamma)'_t&=\sqrt{\la}\varphi(t)+
\sqrt{\la}\int_t^1\overline{R(\gamma)}_u\left(\int_u^1\varphi(s)db^i(s),
\ep_i\right)(\circ db(u))\nonumber\\
&=\sqrt{\la}\varphi(t)+\sqrt{\la}\int_t^1\overline{R(\gamma)}_u\left(\ep_j,
\ep_i\right)(\circ db(u))\int_0^1\varphi^j(s)db^i(s)\nonumber\\
&\quad -\sqrt{\la}\int_t^1\overline{R(\gamma)}_u\left(\int_0^u\varphi(s)db^i(s),
\ep_i\right)(\circ db(u)).
\end{align}
By a standard calculation, we have
\begin{align}
\int_0^1 E\left[|DF(\gamma)'|_t^2\right]dt&\le
\la+C.
\end{align}
This implies (\ref{e2 asymptotics 2}).
\end{proof}

As in the proof above, 
the COH formula and the estimate
$\lim_{\la\to\infty}\|I+R_{0,\la}(\gamma)\|_{op}=1$
immediately implies the lower bound of the 
limit.
In the loop space case, $A(\gamma)_{\la}$ is not uniformly
bounded in $\gamma$ and the existence of the spectral gap is
not obvious.
This difficulty can be solved by using the 
log-Sobolev inequality (\ref{LSI loop}).
In order to obtain precise asymptotics of the spectral gap, 
we need
continuity theorem in rough path analysis.
For this purpose, we need to consider
the operator $A(c_{x_0,y_0})_{\infty}=\lim_{\la\to\infty}A(c_{x_0,y_0})_{\la}$.
In the next section, we study some relations between 
$A(c_{x_0,y_0})_{\infty}$
and the Hessian of the energy function $E$ at $c_{x_0,y_0}$.

\section{Square root of Hessian of the energy function and
Jacobi fields}

In this section, we assume 
$d(x_0,y_0)$ is smaller than the injectivity radius at
$y_0$.
We begin by determining 
$\lim_{\la\to\infty}K(c_{x_0,y_0})_{\la,t}$.
By using (\ref{loghessian 0}), we have
\begin{align}
\lim_{\la\to\infty}K(c_{x_0,y_0})_{\la,t}&=
-\lim_{\la\to\infty}\frac{1}{2\la}\overline{R(c_{x_0,y_0})_t}
+\lim_{\la\to\infty}\frac{1}{\la}\overline{V^{\la}_{y_0}(t,c_{x_0,y_0})_t}
\nonumber\\
&=\frac{1}{1-t}\lim_{\la\to\infty}
\frac{1-t}{\la}\overline{V^{\la}_{y_0}(t,c_{x_0,y_0})_t}\nonumber\\
&=-\frac{1}{1-t}\overline{\nabla^2k(c_{x_0,y_0})_t}.
\end{align}
We write
\begin{align}
K(t)&=-\frac{1} {1-t}\overline{\nabla^2k(c_{x_0,y_0})_t}.
\end{align}
It is natural to conjecture that $A(c_{x_0,y_0})_{\infty}$ is equal to
the operator in $\Ltwo$ given by
\begin{align}
\varphi(t)\mapsto
 \varphi(t)+M(t)^{\ast}\int_t^1M(s)^{\ast}K(s)\varphi(s)ds,
\label{J}
\end{align}
where
$M(t)$ is the solution to
\begin{align}
 M(t)'&=K(t)M(t)\quad \qquad 0\le t<1,\label{M and K}\\
M(0)&=I.
\end{align}
In fact, this is true and we prove it later in more general form
in Lemma~\ref{perturbation of M}.
We study the relation between the operator of (\ref{J})
and $D^2E(c_{x_0,y_0})$.
First, recall that we fix an frame $u_0\in O(M)$ at $x_0$.
Let us choose $\xi\in \RR^n$ so that
$\exp_{x_0}(tu_0(\xi))=\cxy(t)$~$(0\le t\le 1)$, where
$\exp_{x_0}$ stands for the exponential mapping at $x_0$.
Clearly it holds that $d(x_0,y_0)=|\xi|$.
Let $\cyx(t)=\cxy(1-t)$ denote the reverse geodesic path from
$y_0$ to $x_0$.
In order to see the explicit expression of
the Hessian of
$k(z)$~$(z\in \cxy)$, we recall the notion of
Jacobi fields.

Let $R$ be the curvature tensor and define $R(t)=
\overline{R(\cxy)}_t(\cdot,\xi)(\xi)$
which is a linear mapping on $\RR^n$.
Also we define $\Rreverse(t)=R(1-t)$.
Let $v\in \RR^n$ and
$W(t,v)$ be the solution to the following ODE:
\begin{equation}
W''(t,v)+\Rreverse(t)W(t,v)=0~~0\le t\le 1,
~~
W(0,v)=0,~ W'(0,v)=v.
\end{equation}
Since $t\mapsto W(t,v)$ is linear,
let $W(t)$ denote the corresponding 
$n\times n$ matrix.
Of course, $W(0)=0, W'(0)=I$.
Since ${\rm Cut}(y_0)\cap\{\cyx(t)~|~0\le t\le 1\}=\emptyset$,
$W(t)$ is an invertible linear mapping for all $0<t\le 1$ and
$\tilde{W}(t,v)=W(t)W(1)^{-1}v$ is the solution to
$$
\tilde{W}''(t,v)+\Rreverse(t)\tilde{W}(t,v)=0,~~
\tilde{W}(0,v)=0,~\tilde{W}(1,v)=v
$$
and
$(\nabla^2k(\cyx(1))(u_0v,u_0v)=(\tilde{W}'(1,v),\tilde{W}(1,v))
=(W'(1)W(1)^{-1}v,v)$.
This result can be found in many standard books
in differential geometry, {\it e.g.} \cite{jost}.
Let $0<T\le 1$.
We can obtain explicit form of the Jacobi field along
$\cyx(t)$~$(0\le t\le T)$ with given terminal value at $T$ using
$W$.
Let $\tilde{W}_T(t,v)=W(Tt)W(T)^{-1}v$.
Then $\tilde{W}_T(t,v)$~$0\le t\le 1$ satisfies 
the Jacobi equation
\begin{align}
\tilde{W}_T''(t,v)+\Rreverse(tT)T^2\tilde{W}_T(t,v)=0,~~
\tilde{W}_T(0,v)=0,~ \tilde{W}_T(1,v)=v.
\end{align}
Hence
$\nabla^2k(\cyx(t))
\left(\tau(\cxy)_{1-t}u_0v,\tau(\cxy)_{1-t}u_0v\right)
=t\left(W'(t)W(t)^{-1}v,v\right)$.

Next we prove that
$
A(t):=tW'(t)W(t)^{-1}
$
is a symmetric matrix for $0<t\le 1$.
This can be checked by the following argument.
Note that
$W(t)=tI+\int_0^t\int_0^s\int_0^rW'''(u)du$.
This follows from the equation of $W$.
By this observation, 
if we extend $A=A(t)$ by setting $A(0)=I$,
then $A(t)$ is continuously differentiable on
$[0,1]$ and $A'(0)=0$.
We have
\begin{align}
A'(t)&=W'(t)W(t)^{-1}+tW''(t)W(t)^{-1}
-tW'(t)W(t)^{-1}W'(t)W(t)^{-1}\nonumber\\
&=-t\Rreverse(t)-\frac{A(t)^2}{t}+\frac{A(t)}{t}.\label{eq for A}
\end{align}
Let $B(t)=A(t)-A(t)^{\ast}$, where
$A(t)^{\ast}$ denotes the transposed matrix.
Since $\Rreverse(t)$ is a symmetric matrix, 
(\ref{eq for A}) implies
\begin{align}
B(t)&=\frac{1}{t}\int_0^t(I-A(s)^{\ast})B(s)ds+
\frac{1}{t}\int_0^tB(s)(I-A(s))ds,
\qquad 0<t\le 1.
\end{align}
Noting 
\begin{align}
\lefteqn{
\frac{1}{t}\int_0^t(I-A(s)^{\ast})B(s)ds}\nonumber\\
&=
\frac{I-A(t)^{\ast}}{t}\int_0^tB(s)ds
+\frac{1}{t}\int_0^t\left(A(s)^{\ast}\right)'\left(\int_0^sB(r)dr\right)ds
\end{align}
and using Gronwall's inequality, we obtain 
$B(t)=0$ for all $t$ which implies the desired result.

Let $f(t)=W(1-t)$.
Then $f$ satisfies
\begin{align}
f''(t)+R(t)f(t)=0,\qquad 0\le t\le 1,\qquad  f(1)=0,~~ f'(1)=-I.
\end{align}
Since $f'(t)f(t)^{-1}$ is a symmetric matrix,
we have the following key relations:
\begin{align}
& \overline{\nabla^2k(\cxy)}_t=
-(1-t)f'(t)f(t)^{-1}\\
& K(t)=-\frac{1}{1-t}\overline{\nabla^2k(\cxy)}_t=
f'(t)f(t)^{-1}.
\end{align}

Let
\begin{align}
\tilde{K}(t)=K(t)+\frac{1}{1-t}.\label{decomposition of K}
\end{align}
Since
$\tilde{K}(t)=\frac{I-A(1-t)}{1-t}$,
we see that $\tilde{K}(t)$ 
$(0\le t\le 1)$ is a matrix-valued continuous mapping.
Let $N(t)$ be the solution to
$$
N'(t)=\tilde{K}(t)N(t),~N(0)=I.
$$
Then $\sup_t(\|N(t)\|_{op}+\|N^{-1}(t)\|_{op})<\infty$
and $M(t)=(1-t)N(t)$,
where $M(t)$ is the solution to
(\ref{M and K}).
Also we have 
$M(t)=f(t)f(0)^{-1}$.

We write $\Ltwoo=\{\varphi\in \Ltwo~|~\int_0^1\varphi(t)dt=0\}$.
Then
$
\left(U\varphi\right)(t)=\int_0^t\varphi(s)ds
$
is a bijective linear isometry from
$\Ltwoo$ to $\Ho$.
Also $U^{-1}h(t)=\dot{h}(t)$.
Let us introduce an operator
\begin{align}
\left(S\varphi\right)(t)&=\varphi(t)-f'(t)f(t)^{-1}\int_0^t\varphi(s)ds,\\
{\rm D}(S)&=\Ltwoo.
\end{align}
By Hardy's inequality, 
\begin{align}
\int_0^1\left|\frac{1}{1-t}\int_t^1\varphi(s)ds\right|^2dt
\le 4\int_0^1|\varphi(s)|^2ds
\qquad \mbox{for any $\varphi\in \Ltwo$},
\end{align}
we see that $S$ is a bounded linear operator from
$\Ltwoo$ to $\Ltwo$.
The following lemma shows that
$S$ is a square root of the Hessian of
the energy function $E$.
This relation is key to identify the limit
of $\dirichlet$.

\begin{lem}\label{S and T1}
Let $T$ be the bounded linear operator on
$\Ltwoo$ such that
\begin{align}
(T\varphi)(t)=-\int_t^1R(s)\left(\int_0^s\varphi(u)du\right)ds+
\int_0^1\left(\int_t^1R(s)\Bigl(\int_0^s\varphi(u)du\Bigr)ds\right)dt.
\end{align}
Then $T$ is a symmetric operator and for any $\varphi\in \Ltwoo$,
\begin{equation}
\|S\varphi\|^2=
((I+T)\varphi,\varphi),\label{S and T}
\end{equation}
where $I$ denotes the identity operator on $\Ltwoo$.
Moreover,
\begin{equation}
(D_0^2E)(\cxy)=U\left(I+T\right)U^{-1},
\end{equation}
where $E$ is the energy function of the path~$(\ref{energy function})$.
\end{lem}

\begin{proof}
The symmetry of $T$ follows from direct calculation.
Using 
\begin{align*}
 \lim_{t\to 1}\frac{1}{1-t}\left|\int_t^1\varphi(s)ds\right|^2=0,
\quad &
f''(t)=-R(t)f(t),
\quad
\mbox{$f'(t)f(t)^{-1}$ is symmetric},\\
(f(t)^{-1})'&=-f(t)^{-1}f'(t)f(t)^{-1},
\end{align*}
we have
\begin{align}
\|S\varphi\|^2&=\|\varphi\|^2-2\int_0^1\left(
f'(t)f(t)^{-1}\int_0^t\varphi(s)ds,\varphi(t)\right)dt\nonumber\\
& +\int_0^1\left|f'(t)f(t)^{-1}\int_0^t\varphi(s)ds\right|^2dt\nonumber\\
&=\|\varphi\|^2+\int_0^1\left(
\left(f'(t)f(t)^{-1}\right)'\int_0^t\varphi(s)ds,
\int_0^t\varphi(s)ds\right)dt\nonumber\\
&+\int_0^1\left|f'(t)f(t)^{-1}\int_0^t\varphi(s)ds\right|^2dt\nonumber\\
&=
\|\varphi\|^2-\int_0^1
\left(R(t)\int_0^t\varphi(s)ds,\int_0^t\varphi(s)ds\right)
dt\nonumber\\
&=\left((I+T)\varphi,\varphi\right).\label{hessian and R}
\end{align}
By the second variation formula of the energy function
along geodesics (\cite{jost}),
we have
\begin{align*}
 (D_0^2E)(\cxy)(U\varphi,U\varphi)=\left((I+T)\varphi,\varphi\right).
\end{align*} 
Thus the proof is completed.
\end{proof}

Let
\begin{equation}
(S_2\varphi)(t)=\varphi(t)+f'(t)\int_0^tf(s)^{-1}\varphi(s)ds.
\end{equation}
Then again by Hardy's inequality
$S_2$ is a bounded linear operator on $\Ltwo$.
Moreover, it is easy to see that
$\Im(S_2)\subset \Ltwoo$,
$SS_2=I_{\Ltwo}$ and $S_2S=I_{\Ltwoo}$.
Therefore, $S_2=S^{-1}$ and
$\Im(S)=\Ltwo$.
Moreover we have 
$S^{\ast}S=I+T$ on $\Ltwoo$ by (\ref{S and T}).
Note that by identifying the dual space of a Hilbert space
with the Hilbert space itself using Riesz's theorem,
we view $S^{\ast} : (\Ltwo)^{\ast} \to (\Ltwoo)^{\ast}$ as
the operator from $\Ltwo$ to $\Ltwoo$.
We have the following explicit expression
of $S^{-1}$, $S^{\ast}$ and $(S^{-1})^{\ast}$.

\begin{lem}\label{S explicit form}
$(1)$~$S^{-1} : \Ltwo\to \Ltwoo$, 
$S^{\ast} : \Ltwo\to \Ltwoo$ are bijective linear maps and
we have for any $\varphi\in \Ltwo$,
\begin{align}
\left(S^{-1}\varphi\right)(t)&=
\varphi(t)+f'(t)\int_0^tf(s)^{-1}\varphi(s)ds\\
\left(S^{\ast}\varphi\right)(t)&=
\varphi(t)-\int_0^1\varphi(t)dt+\int_0^tf'(s)f(s)^{-1}\varphi(s)ds
\nonumber\\
&-\int_0^1\left(\int_0^tf'(s)f(s)^{-1}\varphi(s)ds\right)dt.
\end{align}

\noindent
$(2)$ $(S^{-1})^{\ast}$ is a bijective linear map 
from $\Ltwoo$ to $\Ltwo$.
If we define $(S^{-1})^{\ast}$ is equal to $0$ on the subset
of constant functions, then for any
$\varphi\in \Ltwo$,
\begin{align}
\left((S^{-1})^{\ast}\varphi\right)(t)&=
\varphi(t)+
\left(f(t)^{\ast}\right)^{-1}\int_t^1
f(s)^{\ast}f'(s)f(s)^{-1}\varphi(s)ds.\label{Sinverseast}
\end{align}
Also $(S^{-1})^{\ast}\varphi$ can be written using $M(t)$ and $K(t)$
as
\begin{equation}
\left((S^{-1})^{\ast}\varphi\right)(t)=
\varphi(t)+
(M(t)^{\ast})^{-1}\int_t^1M(s)^{\ast}K(s)\varphi(s)ds.
\end{equation}
\end{lem}

\begin{proof}
All the calculation are almost similar and so
we show how to calculate $(S^{-1})^{\ast}$ only.
Using $(f'(t)f(t)^{-1})^{\ast}=f'(t)f(t)^{-1}$, 
we have for $\varphi\in \Ltwo$ 
and $\psi\in \Ltwo$,
\begin{align}
& \left(S^{-1}\varphi,\psi\right)_{\Ltwo}\nonumber\\
& \quad=(\varphi,\psi)-\int_0^1
\Bigg\langle\int_0^tf(s)^{-1}\varphi(s)ds, 
\left(
\int_t^1f(s)^{\ast}f'(s)f(s)^{-1}\psi(s)ds\right)'
\Bigg\rangle dt\nonumber\\
&\quad =(\varphi,\psi)+
\int_0^1\Big\langle
\varphi(t), \left(f(t)^{-1}\right)^{\ast}
\int_t^1f(s)^{\ast}f'(s)f(s)^{-1}\psi(s)ds
\Big\rangle dt.
\end{align}
This shows (\ref{Sinverseast})
and $(S^{-1})^{\ast}\mbox{const}=0$.
\end{proof}

We summarize the relation between $S$ and $T$ 
in the proposition below.

\begin{pro}\label{S and T2}
$(1)$~
We have
\begin{align*}
I+T=S^{\ast}S,\quad
(S^{-1})^{\ast}(I+T)=S,\quad
(I+T)^{-1}=S^{-1}(S^{-1})^{\ast}.
\end{align*}

\noindent
$(2)$~The following identities hold.
\begin{align}
\inf\sigma(I+T)=\inf
\left\{\|S\varphi\|^2~|~\|\varphi\|_{\Ltwo}=1,
\varphi\in \Ltwoo\right\}=
\frac{1}{\|(S^{-1})^{\ast}\|_{op}^2}.
\end{align}
\end{pro}

\begin{proof}
$I+T=S^{\ast}S$ follows from Lemma~\ref{S and T1}.
$(I+T)^{-1}=S^{-1}(S^{-1})^{\ast}$ follows from
$(S^{-1})^{\ast}=(S^{\ast})^{-1}$.
(2) follows from (1).
\end{proof}

The identity
$\|S\varphi\|^2=\left((I+T)\varphi,\varphi\right)$
~$(\varphi\in {\Ltwoo})$
is used to prove the upper bound estimate, while
the inequality $\inf\sigma((I+T))\le \frac{1}{\|(S^{-1})^{\ast}\|_{op}^2}$
is used for the proof of the lower bound estimate
in Theorem~\ref{main theorem 1}.
See (\ref{varphiep}) and (\ref{lbd of F}).

\section{Proof of Theorem~\ref{main theorem 1}}\label{proof of main
theorem 1}

We prove Theorem~\ref{main theorem 1}.
So we assume that ${\cal D}$ satisfies conditions (1), (2)
in the theorem
throughout this Section.
As explained already, furthermore,
we may assume $M$ is diffeomorphic to
$\RR^n$ and the Riemannian metric is flat outside a compact set.
Therefore, Assumptions~\ref{assumption A}, \ref{assumption B},
\ref{assumption D} are satisfied.

We consider the ground state function of $L_{\la}$.
Let $\tilde{\chi}_{\delta}(\gamma)=
\chi_{\delta}\Bigl(\max_{0\le t\le 1}d(\gamma(t),\cxy(t))\Bigr)$,
where $\chi_{\delta}$ is a non-negative smooth function such that 
$\chi_{\delta}(u)=1$ for $|u|\le \delta$ and $\chi_{\delta}(u)=0$
for $|u|\ge 2\delta$.
Here $\delta$ is a sufficiently small positive number.
Note that there exists $C_{\delta}>0$ such that
$\nu^{\la}_{x_0,y_0}\left(\max_{0\le t\le 1}d(\gamma(t),\cxy(t))\ge
\delta\right)
\le e^{-\la C_{\delta}}$.
This can be proved by a large deviation result for 
solutions of SDE.
Since the proof is similar to that of
(\ref{rough path exponential decay}),
we omit the proof.

Thus $\|\tilde{\chi}_{\delta}\|_{L^2(\nu^{\la}_{x_0,y_0})}\ge
1-Ce^{-C'\la}$.
Also we have
$
\|D_0\tilde{\chi}_{\delta}\|_{L^2(\nu^{\la}_{x_0,y_0})}
\le Ce^{-C'\la}.
$
Here we have used that the function 
$q(\gamma)=\max_{0\le t\le 1}d(\gamma(t),\cxy(t))$
belongs to ${\rm D}(\E^{\la})$ and
$|D_0q(\gamma)|\le 1$ $\nu^{\la}_{x_0,y_0}$-a.s. $\gamma$.
This is proved in a similar way to Lemma 2.2 (2) in \cite{aida-coh}.
Hence 
\begin{align}
e^{\la}_{Dir,1,{\mathcal D}}\le 
Ce^{-\la C'}.\label{estimate on e1}
\end{align}
On the other hand, it is proved in \cite{aida-coh2} that
$\liminf_{\la\to\infty}\frac{e^{\la}_{Dir,2,{\cal D}}}{\la}>0$.
In \cite{aida-coh2}, we studied the case of compact manifolds.
However, the proof works as well as the present case by the 
assumption on $M$.
These estimates imply 
that $e^{\la}_{Dir,1,{\cal D}}$ is a simple eigenvalue.
Let $\Psi_{\la}$ denote the normalized non-negative
eigenfunction (ground state function).
It is clear that 
$\Psi_{\la}\in H^{1,2}_0\left({\mathcal
D},{\nu}^{\la}_{x_0,y_0}\right)$.
From (\ref{estimate on e1}), we obtain
$\|D_0\Psi_{\la}\|_{L^2(\nu^{\la}_{x_0,y_0})}\le Ce^{-C'\la}$.
It is plausible that $\Psi_{\la}$ is strictly positive for
$\nu^{\la}_{x_0,y_0}$ almost all $\gamma$
which 
follows from the 
positivity improving property of the corresponding
$L^2$-semigroup.
However, we do not need such a property in this paper
and we do not consider such a problem.

We use the following representation of $\dirichlet$
to prove ${\rm LHS}\le {\rm RHS}$ in (\ref{main theorem 1 identity}) in
Theorem~\ref{main theorem 1}.
\begin{align}
e^{\la}_{Dir,2,{\mathcal D}}&=
\inf\Biggl\{
\frac{\int_{{\mathcal D}}
|D_0(F-(\Psi_{\la},F)\Psi_{\la})|^2d\nu^{\la}_{x_0,y_0}}
{\|F-(\Psi_{\la},F)\Psi_{\la}\|_{L^2(\nu^{\la}_{x_0,y_0})}^2}
~\Bigg |
~F\in H^{1,2}_0({\mathcal D})~\nonumber\\
& \qquad\mbox{and}~
\|F-(\Psi_{\la},F)\Psi_{\la}\|_{L^2(\nu^{\la}_{x_0,y_0})}\ne 0
\Biggr\}.\label{representation of e2}
\end{align}

The following estimate is necessary for the proof of Theorem~\ref{main theorem 1}.

\begin{lem}\label{ground state}
We have
\begin{equation}
 \|\Psi_{\la}-1\|_{L^2(P_{x_0,y_0}(M),\nu^{\la}_{x_0,y_0})}\le C e^{-C'\la},
\end{equation}
where $C,C'$ are positive constants.
\end{lem}

\begin{proof}
By the COH formula,
$$
\|\Psi_{\la}-\left(\Psi_{\la},1\right)
_{L^2({\nu}^{\la}_{x,y})}\|_{L^2(\nu^{\la}_{x_0,y_0})}
\le Ce^{-C'\la}.
$$
This implies 
$$
1-\left(\Psi_{\la},1\right)_{L^2({\nu}^{\la}_{x,y})}^2
=\left(\Psi_{\la},\Psi_{\la}-\left(\Psi_{\la},1\right)_{L^2({\nu}^{\la}_{x,y})}\right)
_{L^2(\nu^{\la}_{x_0,y_0})}
\le Ce^{-C'\la}
$$
which shows
$\|\Psi_{\la}-1\|_{L^2(P_{x_0,y_0}(M),\nu^{\la}_{x_0,y_0})}^2\le
2Ce^{-C'\la}$.
\end{proof}

We need the following lemma to prove that
$A(\gamma)_{\la}$ can be approximated by $A(\cxy)_{\infty}(=(S^{-1})^{\ast})$ 
when $\gamma$ is close to $\cxy$ and $\la$ is large.

\begin{lem}\label{perturbation of M}
Recall that we have defined
\begin{align}
 K(t)=-\frac{\overline{\nabla^2k(c_{x_0,y_0})}_t}{1-t}.
\end{align}
We consider a perturbation
of $K(t)$ such that
$$
K_{\ep}(t)=K(t)+\frac{C_{\ep}(t)}{(1-t)^{\delta}},
$$
where $0<\delta<1$ is a constant and
$C_{\ep}(t)$~$(0\le \ep\le 1)$ is a symmetric matrix-valued
continuous function satisfying $\sup_{t}\|C_{\ep}(t)\|\le \ep$.
Let
$M_{\ep}(t)$ be the solution to
\begin{align}
M_{\ep}'(t)&=K_{\ep}(t)M_{\ep}(t)\quad 0\le t<1,\\
M_{\ep}(0)&=I.
\end{align}
Define
\begin{align}
(J_{\ep}\varphi)(t)&=
(M_{\ep}(t)^{\ast})^{-1}\int_t^1
M_{\ep}(s)^{\ast}K_{\ep}(s)\varphi(s)ds.
\end{align}
Then for sufficiently small $\ep$,
there exists a positive constant $C$ which is independent of
$\ep$ such that
\begin{align}
\|J_{\ep}-J_0\|_{op}\le C\ep.\label{difference of J}
\end{align}
\end{lem}

By Lemma~\ref{S explicit form}, we see that
$(S^{-1})^{\ast}=I+J_0$ holds.

\begin{proof}
As already mentioned,
$\tilde{K}(t)=\frac{1}{1-t}+K(t)$ is 
a matrix-valued continuous mapping for $0\le t\le 1$.
Taking this into account, we rewrite
$$
K_{\ep}(t)=-\frac{1}{1-t}+\tilde{K}_{\ep}(t),
$$
where $\tilde{K}_{\ep}(t)=\tilde{K}(t)+\frac{C_{\ep}(t)}{(1-t)^{\delta}}$.
Let $N_{\ep}(t)$ be the solution to
\begin{align}
N_{\ep}(t)'=\tilde{K}_{\ep}(t)N_{\ep}(t)\quad 0\le t<1,\quad
 N_{\ep}(0)=I.
\label{Nept}
\end{align}
Clearly, the solution to this equation exists.
Moreover, $\lim_{t\to 1}N_{\ep}(t)$ exists
and $\sup_{0\le t<1}\|N_{\ep}(t)\|<\infty$.
To see this, we prove the continuity of
$N_{\ep}(t)$ with respect to $t$.
Note that for $0\le s\le t<1$,
\begin{align}
\|N_{\ep}(t)-N_{\ep}(s)\|&\le
\int_s^tC\left(1+\frac{1}{(1-u)^{\delta}}\right)\|N_{\ep}(u)\|du
\nonumber\\
&\le \|N_{\ep}(s)\|C\left((t-s)+
\frac{(1-t)^{1-\delta}-(1-s)^{1-\delta}}{1-\delta}\right)\nonumber\\
&\quad
 +\int_s^tC\left(1+\frac{1}{(1-u)^{\delta}}\right)\|N_{\ep}(u)-N_{\ep}(s)\|
du.
\end{align}
Hence by the Gronwall inequality, we have
\begin{align}
\|N_{\ep}(t)-N_{\ep}(s)\|&\le
\|N_{\ep}(s)\|C\left((t-s)+
\frac{(1-t)^{1-\delta}-(1-s)^{1-\delta}}{1-\delta}\right)
\nonumber\\
&\qquad
\times \exp\left\{
C
\left((t-s)+
\frac{(1-t)^{1-\delta}-(1-s)^{1-\delta}}{1-\delta}\right)
\right\}
\end{align}
which implies the desired result.
Note that $\tilde{K}_0(t)=\tilde{K}(t)$ and
$N_{0}(t)=N(t)$.
Then $M_{\ep}(t)=(1-t)N_{\ep}(t)$.
Also we have $N_{\ep}(t)$ $(0\le t<1)$ is invertible
and
\begin{align*}
N_{\ep}(s)N_{\ep}(t)^{-1}&=N_{\ep}^t(s-t) \qquad 0\le t\le s<1,
\end{align*}
where $N_{\ep}^t(u)$~$(0\le u<1-t)$
is the solution to the equation 
\begin{align*}
\partial_uN^t_{\ep}(u)=\tilde{K}_{\ep}(t+u) N^t_{\ep}(u)
\quad 0\le u<1-t, \qquad N^t_{\ep}(0)=I.
\end{align*}
By a similar calculation to $N_{\ep}$, we have
$\sup_{\ep, t, 0\le u<1-t}\|N^{t}_{\ep}(u)\|<\infty$.
By the definition of $J_{\ep}$, we have
\begin{align}
\left(J_{\ep}\varphi\right)(t)=
\frac{1}{1-t}\int_t^1(1-s)N_{\ep}^t(s-t)^{\ast}K_{\ep}(s)\varphi(s)ds.
\end{align}
Hence by Hardy's inequality,
in order to estimate $J_{\ep}-J_0$, it suffices to estimate
$N_{\ep}^t-N_0^t$.
Note that for $0\le u<1-t$,
\begin{align*}
N_{\ep}^t(u)=N_0^t(u)\left(I+\int_0^u N_0^t(\tau)^{-1}
\frac{C_{\ep}(t+\tau)}{(1-(t+\tau))^{\delta}}N_{\ep}^t(\tau)d\tau\right).
\end{align*}
This and the estimate for $C_{\ep}$ and $N^{t}_{\ep}(u)$ imply
\begin{align*}
 \sup_t|N_{\ep}^t(u)-N_0^t(u)|\le C\ep,
\end{align*}
which completes the
proof of (\ref{difference of J}).
\end{proof}

Let us apply the lemma above in the case where
$K_{\ep}(t)=K(\gamma)_{\la,t}$.
We have
\begin{align}
K(\gamma)_{\la,t}
&=K(t)+
\frac{1}{1-t}
\left(\frac{1-t}{\la}\overline{\nabla^2\log
p\left(\frac{1-t}{\la},y_0,\gamma\right)}_t+
\overline{\nabla^2k(\cxy)}_t\right)-
\frac{1}{2\la}\Ricbar\nonumber\\
&=K(t)+
\frac{1}{1-t}\left(\frac{1-t}{\la}\overline{\nabla^2\log
p\left(\frac{1-t}{\la},y_0,\gamma\right)}_t+
\overline{\nabla^2k(\gamma)}_t\right)\nonumber\\
& \quad+\frac{1}{1-t}\left(\overline{\nabla^2k(\cxy)}_t-
\overline{\nabla^2k(\gamma)}_t\right)
-\frac{1}{2\la}\Ricbar.
\end{align}
Therefore,
\begin{align}
C_{\ep}(t)&=
\frac{1}{(1-t)^{1-\delta}}
\left(
\frac{1-t}{\la}\overline{\nabla^2\log
p\left(\frac{1-t}{\la},y_0,\gamma\right)}_t
+\overline{\nabla^2k(\gamma)}_t\right)\nonumber\\
&\quad +\frac{1}{(1-t)^{1-\delta}}
\left(\overline{\nabla^2k(\cxy)}_t-\overline{\nabla^2k(\gamma)}_t\right)
\nonumber\\
&\quad -\frac{(1-t)^{\delta}}{2\la}\Ricbar.\label{Cept}
\end{align}
We need to show that if $\gamma$ and $c_{x_0,y_0}$
are close enough and $\la$ is large, then
$C_{\ep}(t)$ is small.
Then by Lemma~\ref{perturbation of M},
we obtain that $\|J(\gamma)_{\la}-(S^{-1})^{\ast}\|_{op}$
is small.
Let us check each term of
$C_{\ep}(t)$.
If $\gamma(t)\in B_l(y_0)$ for all $0\le t\le 1$,
the first term converges to $0$ by Lemma~\ref{gong-ma} (1)
as $\la\to \infty$ for $\delta>1/2$.
It is trivial to see that the third term goes to $0$.
Hence, it suffices to prove that
if $\gamma$ and $c_{x_0,y_0}$ is close enough, then
the difference
$\overline{\nabla^2k(\cxy)}_t-\overline{\nabla^2k(\gamma)}_t$
is small.
To this end, we use the results in rough path analysis.

Here, we summarize necessary results from rough path analysis.
The readers are referred to
\cite{lyons98, lq, lcl, friz-victoir, friz-hairer} for rough path analysis.
In Section~\ref{statement},
we define a Brownian motion $b$
with variance $1/\la$ on $\RR^n$ by using the stochastic parallel translation
along $\gamma$ and $b$ is a functional of
$\gamma$.
Conversely, $\gamma$ can be obtained by solving a stochastic
differential equation driven by a Brownian motion $b(t)$.
We may use notation $b_t$ instead of $b(t)$.
From now on, $\mu^{\la}$ denotes the Brownian motion measure with variance
$1/\la$.
We use the notation $\mu$ when $\la=1$.
Let $\{L_i\}_{i=1}^n$ be the canonical horizontal vector fields
and consider an SDE on $O(M)$:
\begin{align}
dr(t,u,b)&=\sum_{i=1}^nL_i(r(t,u,b))\circ db^i(t) \label{horizontal sde}\\
r(0,u,b)&=u\in O(M).
\end{align}
Let $X(t,b)=\pi(r(t,u_0,b))$.
Then the law of $X(\cdot,b)$ coincides with 
$\nu^{\la}_{x_0}$.
Also it holds that 
\begin{align}
\overline{\nabla^2 k(X(b))}_t&=
r(t,u_0,b)^{-1}
(\nabla^2k)(X(t,b))r(t,u_0,b)\qquad \mbox{$\mu^{\la}$-a.s. $b$}.
\label{nabla2 k}
\end{align}
Note that if $b$ is the anti-stochastic development of
the Brownian motion $\gamma(t)$ on $M$, 
then it holds that
$\tau(\gamma)_t=r(t,u_0,b)u_0^{-1}$ $\nu^{\la}_{x_0}$-a.s. $\gamma$.
Since we assume $M$ is diffeomorphic to
$\RR^n$, we have a global coordinate
$x=(x^i)\in \RR^n$
and the Riemannian metric
$g(x)=(g_{ij}(x))$
on the tangent space $T_xM$ which can be identified with
$\RR^n$.
Then the SDE of
$r(t,u_0,b)=(X^i(t,b), e^{k}_l(t,b))$~($e(t,b)=(e^k_l(t,b))\in GL(n,\RR)$)
can be written down explicitly (see \cite{ikeda-watanabe, hsu}) as
\begin{align}
 dX^i(t)&=e^i_j(t)\circ db^j(t) \label{frame bundle sde}\\
de^i_j(t)&=-\sum_{k,l}\Gamma^i_{kl}(X(t))e^{l}_j(t)\circ dX^k(t)
\label{frame bundle sde2}.
\end{align}
Moreover, the coefficients of the
SDE are $C^{\infty}_b$ because
the Riemannian metric is flat outside a certain compact subset.
Therefore we can apply rough path analysis and Malliavin calculus 
to the solution of
the SDE.
Now let us recall the definition of the Brownian rough path.
Let $b(N)$ be the dyadic polygonal approximation of $b$
such that $b(N)_{k2^{-N}}=b_{k2^{-N}}$
and $b(N)_t$ is linear for $k2^{-N}\le t \le (k+1)2^{-N}$
with $0\le k\le 2^{N}-1$.
Define 
$b(N)^1_{s,t}=b(N)_t-b(N)_s$,
$b(N)^2_{s,t}=\int_s^t\left(b(N)_u-b(N)_s\right)\otimes db(N)_u$
for $0\le s\le t\le 1$.
Let $\Omega$ be all elements $b$ belonging to the Wiener space
$W^n$ such that
$b(N)^1_{s,t}$ and $b(N)^2_{s,t}$ converge in the 
Besov type norm $\|\cdot\|_{4m,2\theta}$ and $\|\cdot\|_{2m,\theta}$
respectively (\cite{aida-loop group}).
Here $2/3<\theta<1$ and
$m$ is a sufficiently large positive number.
It is proved in \cite{aida-loop group} that
$\Omega^c$ is a slim set in the sense of Malliavin 
with respect to the Brownian motion measure 
$\mu$.
However, it is easy to check that
the same result holds for
the Brownian motion measure $\mu^{\la}$
with variance $1/\la$ for any $\la>0$.
Moreover, if $b\in \Omega$, then
$b+h\in \Omega$ for any element $h\in \H$.
For $b\in \Omega$, we define $b^1_{s,t}=\lim_{N\to\infty}b(N)^1_{s,t}$
and $b^2_{s,t}=\lim_{N\to\infty}b(N)^2_{s,t}$.
The triple $(1,b^1_{s,t},b^2_{s,t})$ is a
$p$-rough path $(2<p=\frac{2}{\theta}<3)$ and its control function
is given by $\omega(s,t)=C(b)|t-s|$.
$C(b)$ depends on the Besov norm of $b^1$ and $b^2$.
For $h\in \H$, 
we have, $(b+h)^1_{s,t}=b^1_{s,t}+h^1_{s,t}$
and 
$$
(b+h)^2_{s,t}=b^2_{s,t}+h^2_{s,t}+\int_s^t(b_u-b_s)\otimes dh_u+
\int_s^t(h_u-h_s)\otimes db_u.
$$
Note that solutions of rough differential equations driven by
geometric rough paths are smooth.
See Definition 7.1.1 and Corollary 7.1.1 in \cite{lq}.
Therefore, considering the composition of the two maps,
$b(\in \Omega)\mapsto (1,b^1_{s,t},b^2_{s,t})$
and the solution map between geometric rough paths,
we obtain a smooth version $r(t,u_0,b)$ of the solution to 
(\ref{frame bundle sde}) and (\ref{frame bundle sde2}).
Here smooth means 
\begin{enumerate}
 \item the mapping $b(\in \Omega)\mapsto r(t,u_0,b)$ 
is differentiable in the 
$H$-direction and smooth in the sense of Malliavin,
\item the mapping $b(\in \Omega)\mapsto r(t,u_0,b)$ is
      $\infty$-quasi-continuous
(See Theorem 3.2 in \cite{aida-loop group}).
\end{enumerate}
In the terminology of Malliavin calculus,
$r(t,u_0,b)$ is a version of redifinition of the solution
to (\ref{horizontal sde}).

By the uniform ellipticity of (\ref{frame bundle sde}),
we have the following estimate for the Malliavin covariance matrix.
For $p\ge 1$, there exists $p'>0$ such that for large $\la$,
\begin{align}
E[\left\{\det \left(DX(1,b)DX(1,b)^{\ast}\right)\right\}^{-p}] 
\le C\la^{p'}.
\end{align}
Thus the probability measure
$
d\mu^{\la}_{x_0,y_0}=\frac{\delta_{y_0}(X(1,b))d\mu^{\la}}
{c(y_0)p(1/\la,x_0,y_0)}
$
is well-defined,
where $c(y_0)=\sqrt{\det(g_{ij}(y_0))}$ and
$\delta_{y_0}$ denotes Dirac's delta function
on $\RR^n$ and 
$\delta_{y_0}(X(1,b))d\mu^{\la}$ is a generalized Wiener functional
(\cite{watanabe}).
Note that $\mu^{\la}_{x_0,y_0}$ does not charge the slim sets.
Thus the image measure
$X_{\ast}\mu^{\la}_{x_0,y_0}$ is well-defined for
smooth $X(b)$.
Moreover, we have
\begin{align}
\mbox{The joint law of
$(b, \gamma)$ under 
$\nu^{\la}_{x_0,y_0}$}
=\mbox{The joint law of $(b,X(b))$ under
$\mu^{\la}_{x_0,y_0}$}.\label{joint law}
\end{align}
This observation implies that one can use estimates on integration
with respect to (Brownian) rough paths to study the estimate on the
stochastic integrals for the pinned Brownian motion.
In the proof in Section 2, we use cut-off functions
$\chi_{1,\kappa}, \chi_{2,\kappa}$.
In our problem, the existence of such cut-off functions
is not trivial.
The existence of such an appropriate cut-off functions 
are proved in
\cite{aida-semiclassical}.
We use the following result in rough paths.
Below, $r(t,u_0,b)$ may be denoted by $r(t,b)$ for simplicity.

\begin{lem}\label{lemma from rough path}
\noindent
$(1)$~In this statement, we consider the 
smooth version $r(t,b)$ for $b\in \Omega$.
By adopting this version,
a version of $\overline{\nabla^2 k(X(b))}_t$ can be defined 
as $r(t,u_0,b)^{-1}
(\nabla^2k)(X(t,b))r(t,u_0,b)$
which is smooth in the above sense.
Let $l_{\xi}(t)=t\xi$, where
$\xi$ is chosen as
$\exp_{x_0}\left(u_0\xi\right)=y_0$.
Let us define
\begin{align}
 \Xi(b)&=\|b^1\|_{4m,\theta/2}^{4m}+
\|b^2\|_{2m,\theta}^{2m}\qquad b\in \Omega.
\end{align}
Then for any $\ep>0$, there exists $\ep'>0$
such that if 
$\Xi(b-l_{\xi})\le \ep'$
and $X(1,b)=y_0$,
\begin{align}
\left|X(t,b)-\cxy(t)\right|&\le \ep t^{\theta/2} \quad 0\le t\le 1,
\label{rough path estimate1}\\
 \left|
\overline{\nabla^2 k(X(b))}_t-
\overline{\nabla^2 k(X(l_{\xi}))}_t
\right|&\le
\ep (1-t)^{\theta/2}\quad 0\le t\le 1, \label{rough path estimate2}\\
\left|
\sup_{0\le t\le 1}\left|I^2_{t,1}(b)-I^2_{t,1}(l_{\xi})\right|
\right|&\le
 \|\varphi'\|_{\infty}\ep,\label{rough path estimate3}
\end{align}
where 
\begin{align}
I^2(b)_{s,t}=
\int_t^1\overline{R(X(b))}_s\left(\int_s^1\varphi(r)db^i(r),\ep_i\right)
\circ db(s)
\end{align}
and $\varphi\in C^1([0,1],\RR^n)$.
The integral is defined in the sense of rough paths.

\noindent
$(2)$ In this statement, let $b$ be the 
Brownian motion which is obtained by the anti-stochastic development
of the pinned Brownian motion $\gamma$.
Let $\eta$ be a $C^1_b$ function with compact support
on $\RR$.
Let
$\tilde{\eta}(\gamma)=\eta\left(\Xi(b-l_{\xi})\right)$.
Then there exists a constant $C>0$ such that for all $\la\ge 1$
\begin{align}
 |D_0\tilde{\eta}(\gamma)|_{\Ho}\le C
\quad \mbox{for $\nu^{\la}_{x_0,y_0}$-almost all $\gamma$}.
\end{align}
\end{lem}

\begin{proof}
(1)~
(\ref{rough path estimate1}) and (\ref{rough path estimate3})
follow from the fact that
$\cxy(t)=X(t,l_{\xi})$ and the continuity theorem 
for $p$-rough path $(2<p=\frac{2}{\theta}<3)$.
We prove (\ref{rough path estimate2}).
We have
\begin{align}
\lefteqn{\overline{\nabla^2 k(X(b))}_t-
\overline{\nabla^2 k(X(l_{\xi}))}_t}\nonumber\\
&=\left\{
\overline{\nabla^2 k(X(b))}_t-
\overline{\nabla^2 k(X(b))}_1
\right\}-
\left\{
\overline{\nabla^2 k(X(l_{\xi}))}_t
-\overline{\nabla^2 k(X(l_{\xi}))}_1
\right\}\nonumber\\
&\quad
+\overline{\nabla^2 k(X(b))}_1-
\overline{\nabla^2 k(X(l_{\xi}))}_1\nonumber\\
&=\left\{
\overline{\nabla^2 k(X(b))}_t-
\overline{\nabla^2 k(X(b))}_1
\right\}-
\left\{
\overline{\nabla^2 k(X(l_{\xi}))}_t
-\overline{\nabla^2 k(X(l_{\xi}))}_1
\right\},
\end{align}
where we have used $(\nabla^2k)(y_0)=I_{T_{y_0}M}$
and $X(1,b)=\cxy(1)=y_0$.
Hence it suffices to apply the continuity theorem
for $p$-rough path $(2<p=\frac{2}{\theta}<3)$.

\noindent
(2) 
In the case of the derivative $D$, this immediately follows from
Lemma 7.11 in \cite{aida-semiclassical}.
The proof for $D_0$ is the same.
Here we give a sketch of the proof.
Recall that
\begin{align}
(D_0)_{h}b(t)=
h(t)+\int_0^t\int_0^s\overline{R(\gamma)}_u(h(u),\circ db(u))(\circ
 db(s)).
\label{D0b}
\end{align}
We already used this formula
in the proof of Theorem~\ref{e2la on Px0} for the derivative
$D$.
From this formula, we see that
$D_0\left(\Xi(b-l_{\xi})\right)$ 
are given by iterated stochastic integrals of
$b$ and $\gamma$.
By (\ref{joint law}),
we can apply estimates for integration
with respect to the Brownian rough path for $b\in \Omega$.
Thus, the iterated integrals of solutions of rough differential equations
can be estimated by the control function of the
Brownian rough path.
Since the support of $\eta$ is compact, this implies the desired estimate.
\end{proof}

Now, we are ready to prove our first main theorem.

\begin{proof}[Proof of Theorem~$\ref{main theorem 1}$]
First we prove the upper bound estimate.
This will be done by using (\ref{representation of e2})
and choosing appropriate functions $F$ below.
For that purpose, we prepare a large deviation estimate.
Below, several constants depending on parameters
$\kappa, \ep$ appear.
We use the notation $M(x)$ to denote
positive functions of $x$ which may diverge as
$x\to 0$.
On the other hand, we use the notation $C(x)$ to
denote positive functions of $x$ which converge to $0$
as $x\to 0$.
$M(x)$ and $C(x)$ may change line by line.
Let $\eta$ be a non-negative smooth function such that
$\eta(u)=1$ for $u\le 1$ and
$\eta(u)=0$ for $u\ge 2$.
Let $0<\kappa<1$ and set
\begin{align}
{\eta}_{1,\kappa}(\gamma)&=
\eta\left(\kappa^{-1}\Xi(b-l_{\xi})\right),\quad
 {\eta}_{2,\kappa}(\gamma)
=\left\{1-{\eta}_{1,\kappa}(\gamma)^2\right\}^{1/2}.
\end{align}
By (\ref{D0b}) and Lemma~\ref{lemma from rough path} (2),
there exists a positive constant $M(\kappa)$ such that
\begin{align}
|D_0\eta_{1,\kappa}(\gamma)|+
|D_0\eta_{2,\kappa}(\gamma)|\le M(\kappa)
\qquad \nu_{x_0,y_0}^{\la}-a.s. \gamma.
\label{estimate on cut-off}
\end{align}
From (\ref{rough path estimate1}),
for any $\ep>0$, 
$\sup_{0\le t\le 1}|X(t,b)-\cxy(t)|\le \ep$ holds
if $\kappa$ is sufficiently small and $\eta_{1,\kappa}(\gamma)\ne 0$.
Hence $\eta_{1,\kappa}\in H^{1,2}_0({\cal D})$.
Let $\psi$ be a smooth non-negative function on $\RR$
satisfying $\psi(u)=0$ for $u\le \delta_1$
and $\psi(u)=1$ for $u\ge \delta_2$,
where $0<\delta_1<\delta_2$.
Then there exist $C, C'>0$ which depend on $\psi$ 
such that for large $\la$
\begin{align}
E^{\nu^{\la}_{x_0,y_0}}\left[\psi\left(\Xi(b-l_{\xi})\right)
\right]\le C e^{-C'\la}.\label{rough path exponential decay}
\end{align}
We prove this estimate.
Let $B$ be a standard Brownian motion on $\RR^n$.
Since the Wiener functional $B\mapsto
 X\left(1,\frac{B}{\sqrt{\la}}\right)$
is non-degenerate, by using the integration by parts formula
(see \cite{nualart, shigekawa}),
\begin{align}
\lefteqn{E^{\nu^{\la}_{x_0,y_0}}\left[\psi\left(\Xi(b-l_{\xi})\right)
\right]}\nonumber\\
&=\left(c(y_0)p(1/\la,x_0,y_0)\right)^{-1}
E\left[\psi\left(\Xi\left(\frac{B}{\sqrt{\la}}-l_{\xi}\right)\right)\delta_{y_0}
\left(X\left(1,\frac{B}{\sqrt{\la}}\right)\right)\right]\nonumber\\
&=\left(c(y_0)p(1/\la,x_0,y_0)\right)^{-1}
E\left[
\tilde{\psi}\left(\Xi\left(\frac{B}{\sqrt{\la}}-l_{\xi}\right)\right)
G(\ep,\la,B)\phi_{\ep}\left(X\left(1,\frac{B}{\sqrt{\la}}\right)-y_0\right)
\right],
\end{align}
where $\tilde{\psi}, \phi_{\ep}$ are bounded continuous functions on 
$\RR$ and $\RR^n$ respectively such that
$\tilde{\psi}\subset [\delta_1,\infty)$
and ${\rm supp}\, \phi_{\ep}\subset B_{\ep}(0)$.
Also the random variable $G(\la,\ep,B)$ satisfies that for any $p>1$
\begin{align}
 E\left[|G(\la,\ep,B)|^p\right]^{1/p}\le C_{\ep,p}(\la),
\end{align}
where $C_{\ep,p}(\la)$ is a polynomial function of
$\la$.
Let $q=p/(p-1)$.
By the H\"older inequality,
\begin{align}
 E^{\nu^{\la}_{x_0,y_0}}\left[
\psi\left(\Xi(b-l_{\xi})\right)
\right]\le p(1/\la,x_0,y_0)^{-1}
C_{\ep,p}(\la)\mu(A_{\ep})^{1/q},
\end{align}
where
\begin{align}
 A_{\ep}=\left\{B~\Big |~\Xi\left(\frac{B}{\sqrt{\la}}-l_{\xi}\right)\ge
 \delta_1,
\, \,
\left|X\left(1,\frac{B}{\sqrt{\la}}\right)-y_0\right|\le \ep
\right\}.
\end{align}
By the large deviation estimate for Brownian rough path
(\cite{friz-victoir, inahama, lqz}),
we have
\begin{align}
 \limsup_{\la\to\infty}
\frac{1}{\la}
\log \mu\left(A_{\ep}\right)
\le
-\frac{1}{2}\inf\left\{\|h\|_{\H}^2~|~
\Xi\left(h-l_{\xi}\right)\ge \delta_1,\,\,
\left|X(1,h)-y_0\right|\le\ep
\right\}=:J_{\ep}.
\end{align}
For sufficiently small $\ep$,
it holds that $J_{\ep}<-\frac{1}{2}d(x_0,y_0)^2$
which can be proved by a contradiction.
Suppose there exists $h_{\ep}\in \H$ such that
$\lim_{\ep\to 0}\|h_{\ep}\|_{\H}\le d(x_0,y_0)$,
$\Xi\left(h_{\ep}-l_{\xi}\right)\ge \delta_1$
and $\left|X(1,h_{\ep})-y_0\right|\le\ep$.
Let $h_0$ be a weak limit point of $h_{\ep}$.
Then $\|h_0\|_{\H}\le d(x_0,y_0)$.
By Lemma 7.12 in \cite{aida-semiclassical},
$\Xi\left(h_0-l_{\xi}\right)=\lim_{\ep\to
 0}\Xi\left(h_{\ep}-l_{\xi}\right)\ge \delta_1$
and $X(1,h_{0})=\lim_{\ep\to 0}X(1,h_\ep)=y_0$.
By the uniqueness of the minimal geodesic between
$x_0$ and $y_0$,
we have $h_0=l_{\xi}$.
This contradicts $\Xi\left(h_0-l_{\xi}\right)\ge \delta_1$.
Hence there exist $\ep>0$ and $\delta>0$ such that
\begin{align}
 E^{\nu^{\la}_{x_0,y_0}}\left[
\psi\left(\Xi(b-l_{\xi})\right)
\right]\le C_{\ep,p}(\la)p(1/\la,x_0,y_0)^{-1}
\exp\left\{-\la\left(\frac{d(x_0,y_0)^2+\delta}{2q}\right)\right\}.
\end{align}
Since $\lim_{\la\to\infty}\frac{\la^{n/2}\exp\left(-\la d(x_0,y_0)^2/2\right)}
{p(1/\la,x_0,y_0)}$ exists, 
by taking $p$ sufficiently large, this proved the desired inequality.

We now apply (\ref{representation of e2}) to prove the upper bound.
Let us fix a positive number $\ep>0$ and
choose $\varphi_{\ep}\in \Ltwoo\cap C^1([0,1],\RR^n)$ 
with $\|\varphi_{\ep}\|=1$ 
such that
\begin{align}
\sigma_1\le \|S\varphi_{\ep}\|^2\le \|(I+T)\varphi_{\ep}\|\le \sigma_1+\ep.
\label{varphiep}
\end{align}
This is possible because of 
Lemma~\ref{S and T1} and Proposition~\ref{S and T2}.
Note that $\|\varphi_{\ep}'\|_{\infty}$ may diverge when
$\ep\to 0$.
Define
\begin{align}
F_{\ep}(\gamma) 
&=
\sqrt{\la}\left(\int_0^1(\varphi_{\ep}(t),db(t))-
\int_0^1(\varphi_{\ep}(t),\xi)dt\right).
\end{align}
Let $\tilde{F}_{\ep}=F_{\ep}\eta_{1,\kappa}\in H^{1,2}_0({\cal D})$.
We estimate the numerator of the ratio in
(\ref{representation of e2}) for $\tilde{F}_{\ep}$.
Since the Besov norm is stronger than the supremum norm,
we have
\begin{align}
|\tilde{F}_{\ep}(\gamma) |\le C\sqrt{\la}M(\ep)C(\kappa).
\end{align}
By (\ref{D0b})
\begin{align}
\left(D_0F_{\ep}(\gamma),h\right)_{\Ho}
&=\sqrt{\la}\int_0^1\left(\varphi_{\ep}(t),h'(t)\right)dt
+\sqrt{\la}\int_0^1\left(
\varphi_{\ep}(t),\int_0^t\overline{R(\gamma)}_u(h(u),\circ db(u))\circ
 db(t)\right)
\nonumber\\
&=
\sqrt{\la}\int_0^1\left(\varphi_{\ep}(t),h'(t)\right)dt\nonumber\\
&\quad +
\int_0^1\left(
\int_t^1\overline{R(\gamma)}_s\left(\int_s^1\varphi_{\ep}(u)db^i(u),\ep_i
\right)\circ
db(t), h'(t)
\right)dt
\end{align}
and so we have
\begin{align}
D_0F_{\ep}(\gamma)_t'&=
\sqrt{\la}\varphi_{\ep}(t)+
\sqrt{\la}\int_t^1\overline{R(\gamma)}_s\left(
\int_s^1\varphi_{\ep}(r)db^i(r),\ep_i\right)\circ db(s)\nonumber\\
& \quad -\sqrt{\la}\int_0^1\int_t^1\overline{R(\gamma)}_s\left(
\int_s^1\varphi_{\ep}(r)db^i(r),\circ \ep_i\right)(\circ db(s))dt\nonumber\\
&=\sqrt{\la}\varphi_{\ep}(t)-
\sqrt{\la}\int_t^1R(s)\left(
\int_0^s\varphi_{\ep}(u)du\right)ds\nonumber\\
&\quad +\sqrt{\la}\int_0^1\int_t^1R(s)\left(
\int_0^s\varphi_{\ep}(u)du\right)dsdt
+I(\la)_t\nonumber\\
&=\sqrt{\la}(I+T)(\varphi_{\ep})(t)+I(\la)_t,
\end{align}
where $R(s)=\overline{R(\cxy)}_s(\cdot,\xi)(\xi)$ and
$I(\la)_t=(D_0F)(X(\cdot,b))_t'-
(D_0F)(X(\cdot,l_{\xi}))_t'.
$
Note that we have used $\varphi_{\ep}\in \Ltwoo$ in the above.
By (\ref{rough path estimate3}), we have
\begin{align}
\sup_{0\le t\le 1}|I(\la)_t|\le
\sqrt{\la}C(\kappa)M(\ep)
\quad \mbox{if $\eta_{1,\kappa}(\gamma)\ne 0.$}
\label{ICkappa}
\end{align}
Thus we have
\begin{align}
 |D_0\tilde{F}_{\ep}(\gamma)|^2&=
\la|(I+T)\varphi_{\ep}|^2\eta_{1,\kappa}^2+
|I(\la)|^2\eta_{1,\kappa}^2+
2\sqrt{\la}((I+T)\varphi_{\ep},I(\la))\eta_{1,\kappa}^2\nonumber\\
&\quad +
F_{\ep}^2|D_0\eta_{1,\kappa}|^2+
2(D_0F_{\ep},D_0\eta_{1,\kappa})\eta_{1,\kappa}.
\end{align}
By (\ref{rough path exponential decay}) and (\ref{ICkappa}),
we get
\begin{align}
\|D_0\tilde{F}_{\ep}\|_{L^2(\nu^{\la}_{x_0,y_0})}^2
\le \la\|(I+T)\varphi_{\ep}\|_{L^2}^2+
\la C(\kappa)M(\ep)+\la CM(\ep)M(\kappa)e^{-C(\kappa)\la}.\label{upper bound 1}
\end{align}
Combining $\|D_0\Psi_{\la}\|\le Ce^{-C'\la}$,
we obtain
\begin{align}
\|D_0\tilde{F}_{\ep}-
\left(\tilde{F}_{\ep},\Psi_{\la}\right)D_0\Psi_{\la}\|_{L^2(\nu^{\la}_{x_0,y_0})}^2
\le\la\|(I+T)\varphi_{\ep}\|_{L^2}^2+\la C(\kappa)M(\ep)+
\la M(\ep)M(\kappa)e^{-C(\kappa)\la}.\label{Dirichlet norm}
\end{align}
We next turn to the estimate of the denominator in 
(\ref{representation of e2}) for
$\tilde{F}_{\ep}$.
To do so, we use COH formula.
For large $\la>0$, by taking $\kappa$ sufficiently small
and combining Lemma~\ref{gong-ma}, Lemma~\ref{lemma from rough path} (1)
and Lemma~\ref{perturbation of M},
we have
\begin{align}
|(J(\gamma)_{\la}-J_0)(D_0\tilde{F}_{\ep})(\gamma)'|_{L^2(0,1)}&\le
\ep |D_0\tilde{F}_{\ep}(\gamma)'|_{L^2(0,1)}.
\end{align}
Therefore,
using $A(\gamma)_{\la}=I+J(\gamma)_{\la}$,
$(S^{-1})^{\ast}=I+J_0$ and
$(S^{-1})^{\ast}(I+T)=S$,
we have
\begin{align}
& A(\gamma)_{\la}(D_0\tilde{F}_{\ep}(\gamma)')_t\nonumber\\
&=
\left(S^{-1}\right)^{\ast}\left(D_0\tilde{F}_{\ep}(\gamma)'\right)_t
+(J(\gamma)_{\la}-J_0)(D_0\tilde{F}_{\ep}(\gamma)')_t
\nonumber\\
&=
\sqrt{\la}\left(S^{-1}\right)^{\ast}(I+T)\varphi_{\ep}(t)\eta_{1,\kappa}+
(S^{-1})^{\ast}I(\la)_t\eta_{1,\kappa}+
F_{\ep}(\gamma)(S^{-1})^{\ast}\left(D_0\eta_{1,\kappa}\right)'_t\nonumber\\
&\qquad +(J(\gamma)_{\la}-J_0)(D_0\tilde{F}_{\ep}(\gamma)')_t
\nonumber\\
&=
\sqrt{\la}S\varphi_{\ep}(t)+I_2(\la),
\label{COHF1}
\end{align}
and
\begin{align}
\|I_2(\la)\|_{L^2(\nu^{\la}_{x_0,y_0})}&\le
\sqrt{\la}M(\ep)e^{-C(\kappa)\la}+\sqrt{\la}C(\kappa)M(\ep)
+\sqrt{\la}M(\ep)M(\kappa)e^{-C(\kappa)\la}\nonumber\\
&\quad +
\ep\sqrt{\la}\left(C+C(\kappa)M(\ep)+M(\ep)M(\kappa)e^{-C(\kappa)\la}\right).
\label{COHF2}
\end{align}
Since $S\varphi_{\ep}(t)$ is a non-random function,
from (\ref{COHF1}) and (\ref{COHF2})
and the COH formula (\ref{COH loop 1}), we obtain
\begin{align}
\|\tilde{F}_{\ep}-E^{\nu^{\la}_{x_0,y_0}}
[\tilde{F}_{\ep}]\|_{L^2(\nu^{\la}_{x_0,y_0})}^2
&\ge \|S\varphi_{\ep}\|^2-C(\kappa)M(\ep)-
M(\ep)M(\kappa)e^{-C(\kappa)\la}\nonumber\\
&\quad -
\ep\left(C+C(\kappa)M(\ep)+M(\ep)M(\kappa)e^{-C(\kappa)\la}\right).
\label{upper bound 2}
\end{align}
Using Lemma~\ref{ground state},
\begin{align}
\|\tilde{F}_{\ep}-
\left(\tilde{F}_{\ep},\Psi_{\la}\right)\Psi_{\la}\|_{L^2(\nu^{\la}_{x_0,y_0})}^2
&=\|\tilde{F}_{\ep}-(\tilde{F}_{\ep},1)\|_{L^2}^2-2
\left(\tilde{F}_{\ep}-(\tilde{F}_{\ep},\Psi_{\la}), 
(\tilde{F}_{\ep},\Psi_{\la})(\Psi_{\la}-1)\right)
\nonumber\\
&\quad +(\tilde{F}_{\ep},1-\Psi_{\la})^2+(\tilde{F}_{\ep},\Psi_{\la})^2
\|1-\Psi_{\la}\|^2\nonumber\\
&\ge\|\tilde{F}_{\ep}-(\tilde{F}_{\ep},1)\|_{L^2}^2
-M(\ep)M(\kappa)e^{-C'\la}.
\label{L2 norm}
\end{align}
Now we set $\ep$ sufficiently small and next $\kappa$ 
 sufficiently small.
By using the estimates (\ref{Dirichlet norm}), (\ref{upper bound 2}), 
(\ref{L2 norm})
and (\ref{varphiep}), we obtain for large $\la$,
\begin{align}
\frac{\|D_0\tilde{F}_{\ep}-(\tilde{F}_{\ep},
\Psi_{\la})D_0\Psi_{\la}\|^2_{L^2(\nu^{\la}_{x_0,y_0})}}
{\|\tilde{F}_{\ep}-(\tilde{F}_{\ep},\Psi_{\la})\Psi_{\la}\|^2
_{L^2(\nu^{\la}_{x_0,y_0})}}
&\le\frac{\la\|(I+T)\varphi_{\ep}\|_{L^2}^2+\la \ep+
\la M(\ep)M(\kappa)e^{-C(\kappa)\la}}
{\|S\varphi_{\ep}\|_{L^2}^2-C\ep-M(\ep)M(\kappa)e^{-C(\kappa)\la}}\nonumber\\
&
\le\frac{\la(\sigma_1+\ep)^2+\la \ep+
\la M(\ep)M(\kappa)e^{-C(\kappa)\la}}
{\sigma_1-C\ep-M(\ep)M(\kappa)e^{-C(\kappa)\la}}.
\end{align}
This completes the proof of the upper bound.

We next prove lower bound estimate.
Take $F\in H^{1,2}_0({\mathcal D})$
such that $\|F\|_{L^2(\nu^{\la}_{x_0,y_0})}=1$ and
$(F,\eta_{1,\kappa})=0$.
By the IMS localization formula,
\begin{equation}
{\mathcal E}(F,F)=
\sum_{i=1,2}{\mathcal E}(F\eta_{i,\kappa},F\eta_{i,\kappa})
-\sum_{i=1,2}
E^{\nu^{\la}_{x_0,y_0}}[|D_0\eta_{i,\kappa}|^2F^2].
\end{equation}
For any $\ep>0$, by taking $\kappa$ sufficiently small and large $\la$,
by Lemma~\ref{coh formula} (1), Lemma~\ref{perturbation of M}, 
Lemma~\ref{lemma from rough path},
\begin{align}
\lefteqn{
\|F\eta_{1,\kappa}-E^{\nu^{\la}_{x_0,y_0}}[F\eta_{1,\kappa}]\|
_{L^2(\nu^{\la}_{x_0,y_0})}^2}
\nonumber\\
& \quad \le
\frac{\left(\|(S^{-1})^{\ast}\|_{op}+C\ep\right)^2}{\la}
E^{\nu^{\la}_{x_0,y_0}}\left[
|D_0(F\eta_{1,\kappa})|^2
\right].\label{COH Feta}
\end{align}
Thus
we have
\begin{align}
\|F\eta_{1,\kappa}\|_{L^2(\nu^{\la}_{x_0,y_0})}^2
\le
\frac{\left(\|(S^{-1})^{\ast}\|_{op}+C\ep\right)^2}{\la}
E^{\nu^{\la}_{x_0,y_0}}\left[
|D_0(F\eta_{1,\kappa})|^2
\right].\label{Feta1}
\end{align}

Now we estimate the Dirichlet norm of $F\eta_{2,\kappa}$.
The log-Sobolev inequality (\ref{lsi1}) implies that
there exists a positive constant $C$ such that
for any $F\in H^{1,2}_0({\mathcal D})$ and
bounded measurable function $V$ on $P_{x_0,y_0}(M)$,
\begin{equation}
{\mathcal E}(F,F)+E^{\nu^{\la}_{x_0,y_0}}
\left[\la^2 VF^2\right]
\ge
-\frac{\la}{C}\log E^{\nu^{\la}_{x_0,y_0}}
\left[e^{-C\la V}\right]\|F\|_{L^2(\nu^{\la}_{x_0,y_0})}^2.
\label{GNS inequality}
\end{equation}
See Theorem 7 in \cite{gross}.
Also see Lemma~\ref{GNS} in the present paper.
Let $\delta$ be a sufficiently small positive number
and define
$
V(\gamma)=\delta 1_{\eta_{2,\kappa}\ne 0}(\gamma),
$
where $1_A$ denotes the indicator function of a set $A$.
By (\ref{GNS inequality}), there exists $\delta'>0$ such that
\begin{align}
&{\mathcal E}(F\eta_{2,\kappa},F\eta_{2,\kappa})\nonumber\\
&\quad =
{\mathcal E}(F\eta_{2,\kappa},F\eta_{2,\kappa})-\la^2
E^{{\nu}^{\la}_{x_0,y_0}}\left[V(F\eta_{2,\kappa})^2\right]
+\la^2
E^{{\nu}^{\la}_{x_0,y_0}}\left[V(F\eta_{2,\kappa})^2\right]\nonumber\\
&\quad \ge
-\frac{\la}{C} \log E^{{\nu}^{\la}_{x_0,y_0}}
\left[e^{C\la V}\right]\|F\eta_{2,\kappa}\|_{L^2(\nu^{\la}_{x_0,y_0})}^2
+\la^2\delta\|F\eta_{2,\kappa}\|_{L^2(\nu^{\la}_{x_0,y_0})}^2\nonumber\\
&\quad\ge
-\frac{\la}{C}\log\left(1+e^{-\la\delta'}\right)
\|F\eta_{2,\kappa}\|_{L^2(\nu^{\la}_{x_0,y_0})}^2
+\la^2\delta\|F\eta_{2,\kappa}\|_{L^2(\nu^{\la}_{x_0,y_0})}^2\nonumber\\
&\quad\ge (\la^2\delta-C\la e^{-\la\delta'})
\|F\eta_{2,\kappa}\|_{L^2(\nu^{\la}_{x_0,y_0})}^2,\label{Feta2}
\end{align}
where in the third inequality we have used the estimate
(\ref{rough path exponential decay}).

By the estimates (\ref{estimate on cut-off}), (\ref{Feta1}),
(\ref{Feta2}) and the fact that 
$\|F\eta_{1,\kappa}\|_{L^2(\nu^{\la}_{x_0,y_0})}^2
+\|F\eta_{2,\kappa}\|_{L^2(\nu^{\la}_{x_0,y_0})}^2=1$,
we get
\begin{equation}
\E^{\la}(F,F) \ge
\la\min\left(\left(\|(S^{-1})^{\ast}\|_{op}+C\ep\right)^{-2},
\la \delta-Ce^{-\la\delta'}\right)-M(\kappa).\label{lbd of F}
\end{equation}
By the definition of $\dirichlet$, this completes the proof.
\end{proof}

\begin{rem}{\rm
Eberle~{\rm \cite{eberle3}} defined a local spectral gap on
${\cal D}$ by
\begin{align}
 e^{\la}_{E}&=
\inf_{F(\ne 0)\in H^{1,2}_0({\cal D})}
\frac{\int_{{\cal D}}|D_0F|^2d\nu^{\la}_{x_0,y_0}}
{\int_{{\cal D}}\left(
F-\frac{1}{\nu^{\la}_{x_0,y_0}({\cal D})}\int_{{\cal D}}
Fd\nu^{\la}_{x_0,y_0}\right)^2d\nu^{\la}_{x_0,y_0}}.
\end{align}
When ${\cal D}$ satisfies 
conditions (1), (2) in Theorem~\ref{main theorem 1},
the above proof shows also that
\begin{align}
 \lim_{\la\to\infty}\frac{e^{\la}_E}{\la}&=
\sigma_1.
\end{align}
Actually, $e^{\la}_E$ is more related to $e^{\la}_2$ than
$\dirichlet$.
}
\end{rem}

\section{A proof of existence of spectral gap}\label{proof of existence
of spectral gap}

We consider the following setting.
Let $(\Omega,{\mathfrak F},\nu)$ be a probability space
and consider a Dirichlet form $(\E,{\cal F})$ defined on
$L^2(\Omega,\nu)$.
We assume the existence of square field operator $\Gamma$ such that
$$
\E(F,F)=\int_{\Omega}\Gamma(F,F)d\nu,\qquad F\in {\cal F}.
$$
Also we assume $1\in {\cal F}$ and the diffusion property.
That is, for any $\varphi\in C^1_b(\RR)$ and $F\in {\cal F}$,
it holds that $\varphi(F)\in {\cal F}$ and
\begin{align}
 \Gamma(\varphi(F),\varphi(F))=\Gamma(F,F)\varphi'(F)^2.
\end{align}
We write $\Gamma(F)=\Gamma(F,F)$.
We already used the following well known estimate
(\cite{gross}).

\begin{lem}\label{GNS}
 Suppose that for any $F\in {\cal F}$,
\begin{align}
 \int_{\Omega}F(w)^2\log(F(w)^2/\|F\|_{L^2(\nu)}^2)d\nu &\le 
\alpha\E(F,F).\label{LSI}
\end{align}
Then
for any bounded measurable function $V$, we have
\begin{align}
 \E(F,F)+\int_{\Omega}V(w)F(w)^2d\nu(w)&
\ge -\frac{1}{\alpha}\log\left(
\int_{\Omega}e^{-\alpha V(\omega)}d\nu(\omega)\right)
\|F\|_{L^2(\nu)}^2\qquad \mbox{for any $F\in {\cal F}$}.
\end{align}
\end{lem}

Note that in the above lemma,
$(\E,{\cal F})$ is not necessarily a closed form
and the lemma holds for any bilinear form $(\E,{\cal F})$
satisfying the logarithmic Sobolev inequality
(\ref{LSI}).
The spectral gap $e_2$ is defined by
$$
e_2=\inf\left\{
{\cal E}(F,F)~\Big |~\|F\|_{L^2(\nu)}=1, \int_{\Omega}F(w)d\nu(w)=0,
F\in {\cal F}
\right\}.
$$

\begin{thm}\label{main theorem 3}
Let ${\cal F}_0$ be a dense linear subset of
${\cal F}$ with respect to ${\cal E}_1$-norm.
Suppose that there exist positive numbers $\alpha,\beta,r_0$
and $\rho\in {\cal F}$ such that 
$\Gamma(\rho)(w)\le 1$ $\nu$-a.s. $w$ and
\begin{align}
\int_{\Omega}F(w)^2\log(F(w)^2/\|F\|_{L^2(\nu)}^2)d\nu &\le 
\alpha\int_{\Omega}\rho(w)^2\Gamma(F,F)(w)d\nu(w),\quad
\mbox{for all $F\in {\cal F}_0$},\label{LSI general}\\
\nu\left(\rho\ge r\right)&\le e^{-\beta r^2},\qquad  \mbox{for all
$r\ge r_0$}.
\end{align} 
Then
\begin{align}
e_2&\ge
\frac{1}{4}
\min\left(
\frac{1}{8\alpha R(\alpha,\beta,r_0)^2}, ~\frac{\beta}{36\alpha}
\right),
\end{align}
where
\begin{align}
R(\alpha,\beta,r_0)&=
\max\left(\sqrt{\frac{2}{\beta}},~
\frac{192\alpha}{\sqrt{\beta}},~
48\sqrt{\frac{\alpha}{\beta}},
~r_0\right).
\end{align}
\end{thm}

\begin{proof}
Let $R\ge r_0$.
We consider a partition of unity $\{\chi_k\}_{k\ge 0}$
on $[0,\infty)$
such that
\begin{itemize}
\item[(i)] $\chi_k$ is a $C^1$ function,
\item[(ii)] $\chi_0(u)=1$ for $0\le u\le R$
and $\chi_0(u)=0$ for $u\ge 2R$,
\item[(iii)]
${\rm supp}\, \chi_k\subset [Rk, R(k+2)]$\qquad $(k\ge 1)$,
\item[(iv)]
$\sum_{k=0}^{\infty}\chi_k(u)^2=1$ for all $u\ge 0$.
\item[(v)] $\sup_{k,u}|\chi_k'(u)|\le\frac{2}{R}$,
\end{itemize}
Define $\tilde{\chi}_k(w)=\chi_k(\rho(w))$.
Let $F\in {\cal F}_0$ and assume $\|F\|_{L^2(\nu)}=1$
and $\int_{\Omega}F(w)d\nu(w)=0$.
By the IMS localization formula,
we have
\begin{align}
\E(F,F)&=\sum_{k=0}^{\infty}\E(F\tilde{\chi}_k,F\tilde{\chi}_k)-
\sum_{k=0}^{\infty}\int_{\Omega}\Gamma(\tilde{\chi}_k)F^2d\nu
\ge\sum_{k=0}^{\infty}\E(F\tilde{\chi}_k,F\tilde{\chi}_k)
-\frac{8}{R^2}\int_{\rho\ge R}F^2d\nu.\label{IMS1}
\end{align}
We estimate each term $\E(F\tilde{\chi}_k,F\tilde{\chi}_k)$.
First, we estimate $\E(F\tilde{\chi}_0,F\tilde{\chi}_0)$.
We have
\begin{align}
 \left|\int_{\Omega}F(w)\tilde{\chi}_0(w)d\nu(w)\right|&=
\left|\int_{\Omega}F(w)(\tilde{\chi}_0(w)-1)d\nu(w)\right|
\le \nu\left(\rho\ge R\right)^{1/2}\le
e^{-\beta R^2/2}.
\end{align}
The log-Sobolev inequality implies the Poincar\'e inequality and we have
\begin{align}
\E(F\tilde{\chi}_0,F\tilde{\chi}_0)
&\ge \frac{1}{2\alpha R^2}
\left(\|F\tilde{\chi}_0\|_{L^2(\nu)}^2-e^{-\beta
 R^2}\right).\label{lower bound E0}
\end{align}
Next we estimate $\E(F\tilde{\chi}_k,F\tilde{\chi}_k)$ for $k\ge 1$.
Let $\phi_k(w)=1_{[Rk,R(k+2)]}(\rho(w))$ and $\delta>0$.
Then by (\ref{LSI general}) and Lemma~\ref{GNS},
\begin{align}
\E(F\tilde{\chi}_k,F\tilde{\chi}_k)&=
\E(F\tilde{\chi}_k,F\tilde{\chi}_k)-
\int_{\Omega}\delta\phi_k(w)
(F\tilde{\chi}_k)^2(w)d\nu(w)+
\int_{\Omega}\delta\phi_k(w)
(F\tilde{\chi}_k)^2(w)d\nu(w)\nonumber\\
&\ge
-\frac{1}{\alpha R^2(k+2)^2}
\log\left(\int_{\Omega}
e^{\alpha\delta R^2(k+2)^2\phi_k(w)}
d\nu(w)\right)
\|F\tilde{\chi}_k\|_{L^2(\nu)}^2\nonumber\\
&\qquad +
\delta\|F\tilde{\chi}_k\|_{L^2(\nu)}^2.
\end{align}
By the tail estimate of $\rho$, we have
\begin{align}
 \int_{\Omega}
e^{\alpha\delta R^2(k+2)^2\phi_k(w)}d\nu(w)
&\le
1+e^{\alpha\delta R^2(k+2)^2-\beta (Rk)^2}.
\end{align}
Hence
\begin{align}
 \E(F\tilde{\chi}_k,F\tilde{\chi}_k)&\ge
\left(
\delta-
\frac{\exp\left\{\left(\alpha\delta
(k+2)^2-\beta k^2\right)R^2\right\}}
{\alpha R^2(k+2)^2}
\right)\|F\tilde{\chi}_k\|_{L^2(\nu)}^2.
\label{lower bound Ek}
\end{align}

For simplicity, we write 
\begin{align}
 G(\delta,\alpha,\beta,R)&=
\delta-\sup_{k\ge 1}\frac{\exp\left\{\left(\alpha\delta
(k+2)^2-\beta k^2\right)R^2\right\}}
{\alpha R^2(k+2)^2}.
\end{align}
Summing the both sides in the inequalities (\ref{lower bound E0}), 
(\ref{lower bound Ek})
and by using the property (iv),
we obtain the following inequality
\begin{align}
 \E(F,F)&\ge
\min\left(\frac{1}{2\alpha R^2},
G(\delta,\alpha,\beta,R)\right)
-\frac{e^{-\beta R^2}}{2\alpha R^2}
-\frac{8}{R^2}\int_{\rho\ge R}F^2d\nu\label{lower bound Ea}
\end{align}
which is denoted by $I(\delta,\alpha,\beta,R)$.
If $\frac{1}{2\alpha}>8$, this inequality 
with large $R$ and small $\delta$ implies
the existence of spectral gap.
In general, we need more considerations.
Since $\sum_{k=1}^{\infty}\|F\tilde{\chi}_k\|_{L^2(\nu)}^2\ge
 \int_{\rho\ge 2R}F(w)^2d\nu(w)$,
by (\ref{IMS1}) and (\ref{lower bound Ek}),
\begin{align}
 \E(F,F)&\ge
G(\delta,\alpha,\beta,R)
\int_{\rho\ge 2R}F^2(w)d\nu(w)-\frac{8}{R^2}.\label{lower bound Eb}
\end{align}
Let $0\le \ep\le 1$.
Multiplying both sides on the inequality 
$I(\delta,\alpha,\beta,2R)$ by $1-\ep$
and the both sides on (\ref{lower bound Eb}) by $\ep$
and taking summation,
we obtain
\begin{align}
 \E(F,F)&\ge
(1-\ep)\min\left(\frac{1}{8\alpha R^2},
G(\delta,\alpha,\beta,2R)
\right)\nonumber\\
&-\frac{(1-\ep)e^{-4\beta R^2}}{8\alpha R^2}
-\frac{8\ep}{R^2}
+\left(\ep G(\delta,\alpha,\beta,R)-\frac{2(1-\ep)}{R^2}\right)
\int_{\rho\ge 2R}F^2(w)d\nu(w).
\end{align}
Now let $\delta=\frac{\beta}{18\alpha}$.
Then by an elementary calculation,
\begin{align}
G(\delta,\alpha,\beta,R)\ge
\frac{\beta}{18\alpha}-\frac{e^{-\beta R^2/2}}{9\alpha R^2}.
\end{align}
Hence,
if
\begin{align}
 \frac{\beta}{18\alpha}\ge\frac{e^{-\beta R^2/2}}{9\alpha R^2}
+\frac{2(1-\ep)}{R^2\ep}, \label{tail estimate integral}
\end{align}
then
\begin{align}
\E(F,F)&\ge
(1-\ep)\min\left(\frac{1}{8\alpha R^2},
\frac{\beta}{18\alpha}-\frac{e^{-2\beta R^2}}{36\alpha R^2}
\right)
-\frac{(1-\ep)e^{-4\beta R^2}}{8\alpha R^2}
-\frac{8\ep}{R^2}\label{lower bound Ec}
\end{align}
By choosing $\ep,R$ appropriately,
we give a lower bound for $\E(F,F)$.
First, let us choose $\ep$ such that
\begin{align}
 \ep=\min\left(\frac{1}{2},~\frac{1}{512\alpha}\right).
\end{align}
We next choose $R$ such that
\begin{align}
\max\left(\frac{e^{-\beta R^2/2}}{9\alpha R^2},~
\frac{2}{R^2\ep}\right)
&\le \frac{\beta}{36\alpha}.\label{R1}
\end{align}
This condition is equivalent to
\begin{align}
 e^{-\beta R^2/2}\le \frac{1}{4}\beta R^2,\quad 
R^2\ge \frac{72\alpha}{\beta \ep}.
\end{align}
Under this condition, the inequality (\ref{tail estimate integral}) holds and 
by using (\ref{lower bound Ec}), we have
\begin{align}
 \E(F,F)&\ge
\frac{1}{2}\min\left(\frac{1}{8\alpha R^2},~
\frac{\beta}{36\alpha}
\right)
-\frac{e^{-4\beta R^2}}{8\alpha R^2}
-\frac{8\ep}{R^2}.\label{lower bound Ed}
\end{align}
Furthermore, we restrict $R$ so that
\begin{align}
 \max\left(
\frac{e^{-4\beta R^2}}{8\alpha R^2},~
\frac{8\ep}{R^2}
\right)\le 
\frac{1}{8}\min\left(\frac{1}{8\alpha R^2},~
\frac{\beta}{36\alpha}
\right).\label{R2}
\end{align}
This condition is equivalent to
\begin{align}
 e^{-2\beta R^2}\le \frac{1}{8},\quad
e^{-4\beta R^2}\le \frac{\beta R^2}{36},\quad
\ep\le \frac{1}{512\alpha}, \quad 
R^2\ge 
\frac{48^2\alpha}{\beta}\ep.
\end{align}
Thus,
(\ref{R1}) and (\ref{R2}) hold if
\begin{align}
 R\ge\max\left(\sqrt{\frac{2}{\beta}},~
48\sqrt{\frac{\alpha}{\beta}}, ~\sqrt{\frac{72\alpha}{\beta\ep}}\right).
\end{align}
Combining the inequalities (\ref{lower bound Ed})
and (\ref{R2}), we obtain the desired estimate.
\end{proof}

\section{Proof of Theorem~\ref{main theorem 2}}

We prove Theorem~\ref{main theorem 2}
by using the argument in the proof of
Theorem~\ref{main theorem 1} and Theorem~\ref{main theorem 3}.
To this end, we need a tail estimate of
$\rho_{y_0}(\gamma)$.

\begin{lem}\label{main lemma 3}
Let $M$ be an $n$-dimensional rotationally symmetric Riemannian manifold
with a pole $y_0$.
Suppose $\|\varphi'\|_{\infty}<\infty$
and Assumption~{\rm \ref{assumption A}} 
is satisfied.
Let $\la_0>0$.
Let $\rho_{y_0}(\gamma)=1+\max_{0\le t\le 1}d(y_0,\gamma(t))$.
Then there exists a positive constant
$r_0$ which depends on $\varphi$, $\la_0$, $d(x_0,y_0)$ and
the dimension $n$
and a positive constant $C_2$ which depends only on $n$
such that
\begin{align}
\nu^{\la}_{x_0,y_0}\left(\rho_{y_0}(\gamma)\ge r\right)&
\le e^{-C_2\la r^2}
\quad \mbox{for all $r\ge r_0$ and $\la\ge \la_0$}.\label{tail estimate}
\end{align}
\end{lem}

\begin{proof}
Let $z_0$ be a point either $x_0$ or
$y_0$.
Let $X_t$ be the Brownian motion starting at $z_0$ on
$M$ 
whose generator is $\Delta/(2\la)$.
First, we give a tail estimate on 
$\rho_{y_0}$ with respect to 
$\nu^{\la}_{z_0}$.
Let $Y_t=d(X_t,y_0)$.
Note that $\Delta_x
 d(x,y_0)=(n-1)\left(\frac{1}{d(x,y_0)}+\varphi'(d(x,y_0))\right)$
and $|\nabla_x d(x,y_0)|=1$.
By the It\^o formula, we have
\begin{align}
Y_t=
d(z_0,y_0)+\frac{1}{\sqrt{\la}}B_t
+\int_0^t\frac{n-1}{2\la}
\left(\frac{1}{Y_s}+\varphi'(Y_s)\right)ds.
\label{sde1}
\end{align}
Here $B_t$ is $1$-dimensional standard Brownian motion.
We can rewrite this equation as
\begin{align}
\sqrt{\la}Y_t&=
\sqrt{\la}d(z_0,y_0)+B_t+
\frac{n-1}{2\sqrt{\la}}\|\varphi'\|_{\infty}t+
\int_0^t\frac{n-1}{2\sqrt{\la}Y_s}ds\nonumber\\
&\quad +\int_0^t\frac{n-1}{2\sqrt{\la}}
\left(\varphi'(Y_s)-\|\varphi'\|_{\infty}\right)ds.
\label{sde2}
\end{align}
Let $\tilde{Z}_t$ be the strong solution to the 
SDE:
\begin{align}
\tilde{Z}_t&=
\sqrt{\la}d(z_0,y_0)+
B_t+\frac{n-1}{2\sqrt{\la}}\|\varphi'\|_{\infty}t+
\int_0^t\frac{n-1}{2\tilde{Z}_s}ds,\label{sde3}
\end{align}
where $B_t$ is the same Brownian motion as in
(\ref{sde2}).
Then by the comparison theorem of 1 dimensional SDE
(see Chapter \Roman{fff} in \cite{ikeda-watanabe}),
we see
\begin{align}
 \sqrt{\la}Y_t\le \tilde{Z}_t\qquad t\ge 0.
\end{align}
Let us define
 $\hat{Z}_t=\tilde{Z}_t-\frac{n-1}{2\sqrt{\la}}\|\varphi'\|_{\infty}t$.
Then $\hat{Z}_t$ satisfies the SDE
\begin{align}
 \hat{Z}_t=\sqrt{\la}d(z_0,y_0)+\int_0^t
\frac{n-1}{2}\frac{1}{\hat{Z}_s+
\frac{n-1}{2\sqrt{\la}}\|\varphi'\|s}ds
+B_t.
\end{align}
Now consider the $n-1$ dimensional
Bessel process $Z_t$ as the strong solution of the SDE:
\begin{align}
Z_t=\sqrt{\la}d(z_0,y_0)+\int_0^t
\frac{n-1}{2}\frac{1}{Z_s}ds
+B_t.
\end{align}
Again by the comparison theorem, we have
\begin{align}
 \hat{Z}_t\le Z_t \qquad t\ge 0.
\end{align}
The law of $\{Z_t\}_{t\ge 0}$ is the same as
the law of $\{|B^{(n)}_t+\sqrt{\la}d(z_0,y_0){\bf e}|\}$,
where $B^{(n)}$ is the standard Brownian motion starting at $0$ and
${\bf e}$ is the unit vector in $\RR^n$.
Thus, for any $r>0$, we have
\begin{align}
 P\left(\max_{0\le t\le 1}Y_t\ge r\right)
&\le P\left(\max_{0\le t\le 1}|B^{(n)}_t+\sqrt{\la}
d(z_0,y_0){\bf e}|+\frac{n-1}{2\sqrt{\la}}\|\varphi'\|_{\infty}\ge
 \sqrt{\la}r\right)\nonumber\\
&\le P\left(\max_{0\le t\le 1}|B^{(n)}_t|\ge 
\sqrt{\la}\left(r-d(z_0,y_0)-
\frac{n-1}{2\la}\|\varphi'\|_{\infty}\right)\right).
\end{align}
Let $C_n=E[\max_{0\le t\le 1}|B^{(n)}_t|]$.
Then there exists $C>0$ such that for any $r>C_n$,
\begin{align}
P\left(\max_{0\le t\le 1}|B^{(n)}_t|\ge r\right)
\le C \exp\left(-\frac{1}{2}(r-C_n)^2\right).
\end{align}
Hence, if
$r>d(z_0,y_0)+\frac{n-1}{2\la}\|\varphi'\|_{\infty}+\frac{C_n}{\sqrt{\la}}$,
then
\begin{align}
 P\left(\max_{0\le t\le 1}
Y_t\ge r\right)&\le
C\exp\left[
-\frac{\la}{2}
\left(r-d(z_0,y_0)-\frac{n-1}{2\la}\|\varphi'\|_{\infty}
-\frac{C_n}{\sqrt{\la}}\right)^2\right].
\end{align}
This shows that there exists $r_0>0$ which depends only on
$d(z_0,y_0), \la_0$ and a positive constant $C$
such that
\begin{align}
 \nu^{\la}_{z_0}\left(\rho_{y_0}(\gamma)\ge r\right)\le
e^{-\la C r^2}\quad \mbox{for all $r\ge r_0$}.\label{tail estimate for rho1}
\end{align}
The tail estimate for $\nu^{\la}_{x_0,y_0}$
can be proved by using 
the absolute continuity 
of $\nu^{\la}_{x_0,y_0}$
with respect to 
$\nu^{\la}_{x_0}$
up to time $t<1$.
The density is given by
\begin{align}
 \frac{d\nu^{\la}_{x_0,y_0}}{d\nu^{\la}_{x_0}}(\gamma)
\Big |_{{\mathfrak F}_t}=
\frac{p\left(\frac{1-t}{\la},y_0,\gamma(t)\right)}
{p\left(\frac{1}{\la},y_0,x_0\right)}=\varphi_{x_0,y_0}(t,\gamma).
\end{align}
Recall that Gaussian upper bound holds
for all $0<t\le 1$ and $x,y\in M$,
\begin{align}
 p(t,x,y)&\le C't^{-n/2}e^{-C''d(x,y)^2/t}.
\end{align}
By Varadhan's heat kernel estimate,
for any $\ep>0$, we have for sufficiently large $\la$,
\begin{align}
p(1/\la,y_0,x_0)\ge
e^{-\la\frac{d(y_0,x_0)^2+\ep}{2}}.
\label{heat kernel estimate}
\end{align}
By using these estimates,
we obtain
\begin{align}
\varphi_{x_0,y_0}\left(\frac{1}{2},\gamma\right)\le
C' \la^{n/2}e^{\frac{\la}{2}\left(d(x_0,y_0)^2+\ep\right)}.
\label{density estimate}
\end{align}
This estimate and (\ref{tail estimate for rho1}) implies that
\begin{align}
 \nu^{\la}_{x_0,y_0}\left(1+\max_{0\le t\le 1/2}d(y_0,\gamma(t))\ge r\right)
\le C' \la^{n/2}e^{\frac{\la}{2}\left(d(x_0,y_0)^2+\ep\right)-\la Cr^2}
\quad \mbox{for all $r\ge r_0$}
\end{align}
Since 
\begin{align}
\nu^{\la}_{x_0,y_0}\left(1+\max_{1/2\le t\le 1}d(y_0,\gamma(t))\ge
 r\right)
=
\nu^{\la}_{y_0,x_0}\left(1+\max_{0\le t\le 1/2}d(y_0,\gamma(t))\ge r\right),
\end{align}
using (\ref{tail estimate for rho1}) with $z_0=y_0$,
similarly,
we obtain the desired tail estimate for $\rho_{y_0}$ under
$\nu^{\la}_{x_0,y_0}$.
\end{proof}

\begin{proof}[Proof of Theorem~$\ref{main theorem 2}$]
Let $\la_0>0$ and consider a positive number
$\la\ge \la_0$.
By Lemma~\ref{main lemma 3},
the assumptions in Theorem~\ref{main theorem 3} are valid for
$\rho=\rho_{y_0}$, $\alpha=C_1/\la, \beta=C_2\la$ and $r_0$.
Hence Theorem~\ref{main theorem 3} implies $e^{\la}_2>0$
for all $\la>0$.
We need to prove the asymptotic behavior
(\ref{asymptotics of ela2}).
We argue similarly to the proof of Theorem~\ref{main theorem 3}.
That is, we use the same functions there and choose
$R, \delta, \ep$ which were defined there.
Let $F\in \FC(P_{x_0,y_0}(M))$
and assume $\|F\|_{L^2(\nu^{\la}_{x_0,y_0})}=1$ and
$E^{\nu^{\la}_{x_0,y_0}}\left[F\right]=0$.
Then by the IMS localization formula
$(\ref{IMS1})$,
we get
\begin{align}
 \E^{\la}(F,F)&\ge
\E^{\la}(F\tilde{\chi}_0,F\tilde{\chi}_0)+
(C\la^2-C'\la)\sum_{k=1}^{\infty}\|F\tilde{\chi}_k\|_{L^2}^2
-\frac{8}{R^2}.
\end{align}
Next we estimate $\E^{\la}(F\tilde{\chi}_0,F\tilde{\chi}_0)$.
Since this is a local estimate, we may vary the Riemannian metric
so that the metric is flat outside certain compact subset.
Take the same function $\eta_{1,\kappa}, \eta_{2,\kappa}$
as in the proof of the lower bound estimate in
Theorem~\ref{main theorem 1}.
Then by the estimate (\ref{rough path exponential decay}),
$|E^{\nu^{\la}_{x_0,y_0}}[F\tilde{\chi}_0\eta_{1,\kappa}]|
\le Ce^{-\la C}$.
In a similar way to the proof of the lower bound in
Theorem~\ref{main theorem 1},
we obtain
\begin{align*}
 \E^{\la}(F\tilde{\chi}_0,F\tilde{\chi}_0)\ge \la
\min\left(\left(\left(\|(S^{-1})^{\ast}\|_{op}+C\ep\right)^{-2},
\la \delta-Ce^{-\la\delta'}\right)\right)
\|F\tilde{\chi}_0\|_{L^2(\nu^{\la}_{x_0,y_0})}^2
-Ce^{-\la C}-M(\kappa).
\end{align*}
Combining the above, the proof of the lower bound is completed.
The upper bound estimate immediately follows from the
estimate (\ref{upper bound 1}) and (\ref{upper bound 2}).
\end{proof}

\noindent
{\bf Acknowledgement}

\noindent
This research was partially supported by Grant-in-Aid for
Scientific Research (B) No.24340023.
The author would like to thank referees for their 
valuable comments and suggestions which improve the 
quality of the paper.


\begin{thebibliography}{99}
%
\bibitem{aida-irred}
S.~Aida,
On the irreducibility of certain Dirichlet forms on loop spaces over 
compact homogeneous spaces. 
New trends in stochastic analysis 
(Charingworth, 1994), 3–42, 
World Sci. Publ., River Edge, NJ, 1997.
%
\bibitem{aida-gradient}
S.~Aida,
Gradient estimates of harmonic functions and the asymptotics
of spectral gaps on path spaces,
Interdisciplinary Information Sciences,
Vol.2, No.1, 75--84 (1996).
%
\bibitem{aida-coh}
S.~Aida,
Logarithmic derivatives of heat kernels and logarithmic Sobolev inequalities 
with unbounded diffusion coefficients on loop spaces.  
J. Funct. Anal.  174  (2000),  no. 2, 430--477. 
%
\bibitem{aida-semiclassical0}
S.~Aida,
Semiclassical limit of the lowest eigenvalue of a Schr\"odinger operator 
on a Wiener space, J. Funct. Anal. 203 (2003), no.2, 401--424. 
%
\bibitem{aida-precise}
S.~Aida,
Precise Gaussian estimates of heat kernels on asymptotically flat
Riemannian manifolds with poles, 
in "Recent developments in stochastic analysis and related topics", 
Proceedings of the First Sino-German conference on stochatsic analysis 
1--19, 2004.
%
\bibitem{aida-semiclassical}
S.~Aida,
Semi-classical limit of the bottom of spectrum of a Schr\"odinger operator on 
a path space over a compact Riemannian manifold.  
J. Funct. Anal.  251  (2007),  no. 1, 59--121. 
%
\bibitem{aida-coh2}
S.~Aida,
COH formula and Dirichlet Laplacians
on small domains of pinned path spaces,
Contemporary Mathematics, 545,
Amer.Math. Soc.,Providence, RI, 2011, 1-12.
%
\bibitem{aida-loop group}
S.~Aida,
Vanishing of one dimensional $L^2$-cohomologies of
loop groups, 
J.Funct.Anal. 261 (2011), no.8, 2164-2213.
%
\bibitem{andersson-driver}
L.~Andersson and B.~Driver,
Finite-dimensional approximations to Wiener measure and path integral
formulas on manifolds. J. Funct. Anal. 165 (1999), no. 2, 430–498. 
%
\bibitem{besse}
A.~Besse,
Einstein manifolds. Ergebnisse der Mathematik und ihrer Grenzgebiete, 
10. Springer-Verlag, Berlin, 1987.
%
\bibitem{chl}
M.~Capitaine, E.~Hsu and M.~Ledoux,
Martingale representation and a simple proof of 
logarithmic Sobolev inequalities on path spaces.  
Electron. Comm. Probab.  2  (1997), 71--81 
%
\bibitem{cgg}
P.~Cattiaux, I.~Gentil and A.~Guillin,
Weak logarithmic Sobolev inequalities and entropic convergence.  
Probab. Theory Related Fields  139  (2007),  no. 3-4, 563--603.
%
\bibitem{clw1}
X.~Chen, X.-M.~Li and B.~Wu,
A Poincar\'e inequality on loop spaces.  J. Funct. Anal.  259  (2010),  
no. 6, 1421–1442.
%
\bibitem{clw11}
X.~Chen, X.-M.~Li and B.~Wu,
A spectral gap for the Brownian bridge measure on hyperbolic spaces.
Progress in analysis and its applications,
398--404, World Sci. Publ., 2010.
%
\bibitem{clw2}
X.~Chen, X.-M.~Li and B.~Wu,
A concrete estimate for the weak Poincar\'e
inequality on loop space,
Probab.Theory Relat. Fields 151 (2011), no.3-4, 559-590.
%
\bibitem{chow}
B.~Chow and S-C, Chu, {\it etal.},
The Ricci flow: techniques and applications. 
Part I. Geometric aspects. Mathematical Surveys and Monographs, 
135. American Mathematical Society, Providence, RI, 2007.
%
\bibitem{cruzeiro-malliavin}
A.B.~Cruzeiro and P.~Malliavin,
Renormalized differential geometry on path space: 
structural equation, curvature. J. Funct. Anal. 139 (1996), no. 1, 119–181. 
%
\bibitem{driver0}
B.K.~Driver,
A Cameron-Martin type quasi-invariance theorem for 
Brownian motion on a compact Riemannian manifold. 
J. Funct. Anal. 110 (1992), no. 2, 272–376. 
%
\bibitem{driver01}
B.K.~Driver,
A Cameron-Martin type quasi-invariance theorem for 
pinned Brownian motion on a compact Riemannian manifold,
Trans. Amer. Math. Soc. 342 (1994), no. 1, 375–395.
\bibitem{driver1}
B.K.~Driver,
The non-equivalence of Dirichlet forms on path spaces. 
Stochastic analysis on infinite-dimensional spaces 
(Baton Rouge, LA, 1994), 75–87, 
Pitman Res. Notes Math. Ser., 310, Longman Sci. Tech., Harlow, 1994. 
%
\bibitem{eberle1}
A.~Eberle, 
Absence of spectral gaps on a class of loop spaces.  
J. Math. Pures Appl. (9)  81  (2002),  no. 10, 915--955. 
%
\bibitem{eberle2}
A.~Eberle,
Spectral gaps on discretized loop spaces.  
Infin. Dimens. Anal. Quantum Probab. Relat. Top.  6  (2003),  no. 2, 265--300.
%
\bibitem{eberle3}
A.~Eberle,
Local spectral gaps on loop spaces.  
J. Math. Pures Appl. (9)  82  (2003),  no. 3, 313--365.
%
\bibitem{elworthy-li1}
K.D.~Elworthy and Xue-Mei Li,
It\^o maps and analysis on path spaces. 
Math. Z. 257 (2007), no. 3, 643–706. 
%
\bibitem{enchev-stroock1}
O.~Enchev and D.W.~Stroock, 
Integration by parts for pinned Brownian motion. 
Math. Res. Lett. 2 (1995), no. 2, 161–169.
%
\bibitem{enchev-stroock2}
O.~Enchev and D.W.~Stroock, 
Pinned Brownian motion and its perturbations. 
Adv. Math. 119 (1996), no. 2, 127–154. 
%
\bibitem{fang}
S.~Fang,
In\'egalit\'e du type de Poincar\'e sur l'espace des chemins
riemanniens. C. R. Acad. Sci. Paris S\'er. I Math. 318 (1994), no. 3, 257–260.
%
\bibitem{friz-hairer}
P.Friz and M.~Hairer,
A course on rough paths. 
With an introduction to regularity structures. 
Universitext. Springer, 2014.
%
\bibitem{friz-victoir}
P.~Friz and N.~Victoir,
Multidimensional stochastic processes as rough paths. 
Theory and applications. 
Cambridge Studies in Advanced Mathematics, 120. 
Cambridge University Press, Cambridge, 2010. 
%
\bibitem{gong-ma}
F.~Gong and Z.~Ma,
The log-Sobolev inequality on loop space over a compact Riemannian manifold.  
J. Funct. Anal.  157  (1998),  no. 2, 599--623. 
%
\bibitem{gordina}
M.~Gordina,
Quasi-invariance for the pinned Brownian motion on a
Lie group,
Stochastic Process. Appl. Vol. 104 (2003), 243--257.
%
\bibitem{greene-wu}
R.E.~Greene and H.~Wu,
Function thorey on manifolds which possess a pole,
Lecture Notes in Mathematics, 699 (1979), Springer, Berlin.
%
\bibitem{gross}
L.~Gross,
Logarithmic Sobolev inequalities. Amer. J. Math. 97 
(1975), no. 4, 1061–1083. 
%
\bibitem{hf}
B.~Helffer and F.~Nier,
Quantitative analysis of metastability in 
reversible diffusion processes via a Witten complex approach: 
the case with boundary. Mém. Soc. Math. Fr. (N.S.) No. 105 (2006),
%
\bibitem{hks}
R.~Holley, S.~Kusuoka and D.W.~Stroock,
Asymptotics of the spectral gap with applications to the theory of
	simulated annealing. 
J. Funct. Anal. 83 (1989), no. 2, 333–347. 
%
\bibitem{hsu}
E.~Hsu,
Stochastic analysis on manifolds. 
Graduate Studies in Mathematics, 38. 
American Mathematical Society, Providence, RI, 2002.
%
\bibitem{hsu1}
E.~Hsu,
Quasi-invariance of the Wiener measure on path spaces: 
noncompact case, J. Funct. Anal. 193 (2002), no. 2, 278–290. 
%
\bibitem{ikeda-watanabe}
N.~Ikeda and S.~Watanabe,
Stochastic differential equations and diffusion processes.
Second edition. North-Holland Mathematical Library, 24, 1989.
%
\bibitem{inahama}
Y.~Inahama,
Large deviation principle of Freidlin-Wentzell type for pinned diffusion processes,
to appear in Trans. Amer. Math. Soc., 35 pages. 
arXiv:1203.5177
%
\bibitem{jost}
J.~Jost,
Riemannian geometry and geometric 
analysis,
second edition, Springer, 1998.
%
\bibitem{laetsch}
T.~Laetsch,
An approximation to Wiener measure and quantization of the 
Hamiltonian on manifolds with non-positive sectional curvature.
J. Funct. Anal. 265 (2013), no. 8, 1667--1727. 
%
\bibitem{leandre1}
R.~L\'eandre,
Integration by parts formulas and rotationally invariant 
Sobolev calculus on free loop spaces,
J. Geom. Phys. 11 (1993), no. 1-4, 517--528.
%
\bibitem{lqz}
M.~Ledoux, Z.~Qian and T.~Zhang,
Large deviations and support theorem for diffusions
via rough paths,
Stochastic proccess and their applications,
102, No.2 (2002), 265--283.
%
\bibitem{li-yau}
P.~Li and S.T.~Yau,
On the parabolic kernel of the Schr\"odinger operator. Acta Math. 156
(1986), no. 3-4, 153–201. 
%
\bibitem{lim}
A.~Lim,
Path integrals on a compact manifold
with non-negative curvature,
Reviews in Mathematical Physics,
Vol. 19, No. 9 (2007), 967--1044.
%
\bibitem{lyons98}
T.~Lyons,~
Differential equations driven by rough signals,
Rev.Mat.Iberoamer., 14 (1998), 215-310.
%
\bibitem{lcl}
T.~Lyons, M.~Caruana and T.~L\'evy,
Difefrential equations driven by rough paths,
Ecole d'Et\'e de Probabilit\'es de Saint-Flour 
\Roman{xxxiv}-2004, 
Lecture Notes in Mathematics, 1908,
Springer-Verlag Berlin Heiderberg
2007.
%
\bibitem{lq}
T.~Lyons and Z.~Qian,
System control and rough paths,
Oxford Mathematical Monographs,
2002.
%
\bibitem{ms}
P.~Malliavin and D.W.~Stroock,
Short time behavior of the heat kernel and its logarithmic derivatives.  
J. Differential Geom.  44  (1996),  no. 3, 550--570.
%
\bibitem{nualart}
D.~Nualart,
The Malliavin calculus and related topics,
Second edition. Probability and its Applications,
Springer-Verlag, Berlin, 2006.
%
\bibitem{sasamori}
T.~Sasamori,
On estimates of heat kernels on Riemannian manifolds with poles,
Master Thesis in 2015, March.
%
\bibitem{simon}
B.~Simon,
Semiclassical Analysis of Low
Lying Eigenvalues~I. Nondegenerate Minima:
Asymptotic Expansions,
Ann.~Inst.~Henri Poincar\'e, Section~A,
Vol.~\Roman{x}, no.~4, (1983), 295--308.
%
\bibitem{shigekawa}
I.~Shigekawa,
Stochastic analysis,
Translations of Mathematical Monographs, 224. 
Iwanami Series in Modern Mathematics,
American Mathematical Society, Providence, RI, 2004.
%
\bibitem{stroock}
D.W.~Stroock,
An estimate on the Hessian of the heat kernel. 
It\^o's stochastic calculus and probability theory, 
355--371, Springer, Tokyo, 1996. 
%
\bibitem{watanabe}
S.~Watanabe,
Analysis of Wiener functionals (Malliavin calculus) and its applications
to heat kernels,
Ann. of Probab. Vol. 15, No.1, (1987), 1--39.

\end{thebibliography}
\end{document}